\documentclass[11pt]{amsart}

\usepackage{amsmath}
\usepackage{amssymb}
\usepackage{amsthm}
\usepackage{mathrsfs}
\usepackage{enumitem}
\usepackage[all]{xypic}
\usepackage[bookmarks=true,bookmarksnumbered=true]{hyperref}
\usepackage[hmargin=2cm,vmargin=3cm]{geometry}

\numberwithin{equation}{section}

\theoremstyle{plain}
\newtheorem{thm}{Theorem}[section]

\newtheorem{prop}[thm]{Proposition}

\newtheorem*{GPconj}{Gan--Gross--Prasad conjecture}
\newtheorem*{main}{Main Theorem}

\theoremstyle{definition}

\theoremstyle{remark}
\newtheorem{rem}[thm]{Remark}

\def\disc{\operatorname{disc}}

\def\Hom{\operatorname{Hom}}

\def\Irr{\operatorname{Irr}}

\def\N{\operatorname{N}}

\def\Tr{\operatorname{Tr}}

\def\As{\mathrm{As}}
\def\GL{\mathrm{GL}}

\def\SL{\mathrm{SL}}

\def\U{\mathrm{U}}

\def\WD{\mathit{WD}}

\def\CC{\mathbb{C}}

\def\ZZ{\mathbb{Z}}

\def\1{\Eins}

\makeatletter
\def\varddots{\mathinner{\mkern1mu
    \raise\p@\hbox{.}\mkern2mu\raise4\p@\hbox{.}\mkern2mu
    \raise7\p@\vbox{\kern7\p@\hbox{.}}\mkern1mu}}
\makeatother

\newcommand{\BIGOP}[1]{\mathop{\mathchoice%
{\raise-0.22em\hbox{\huge $#1$}}%
{\raise-0.05em\hbox{\Large $#1$}}{\hbox{\large $#1$}}{#1}}}

\newcommand{\BIGboxplus}{\mathop{\mathchoice%
{\raise-0.35em\hbox{\huge $\boxplus$}}%
{\raise-0.15em\hbox{\Large $\boxplus$}}{\hbox{\large $\boxplus$}}{\boxplus}}}

\title{The local Gan-Gross-Prasad conjecture for $U(3) \times U(2)$ : the non-generic case}
\author{Jaeho Haan}
\address{Algebraic Structure and its Applications Research Center(ASARC), Department
of Mathematics, Korea Advanced Institute of Science and Technology}
\email{jaehohaan@gmail.com}
\keywords{Gan-Gross-Prasad conejcture, non-tempered Arthur packet, unitary groups, local theta correspondence, $\epsilon$-factor\\ \noindent This work was supported by the National Research Foundation of Korea(NRF) grant funded by the Korea government(MSIP)(ASARC, NRF-2007-0056093).
}
\date{\today}

\begin{document}

\begin{abstract}In this paper, we investigate the local Gan-Gross-Prasad conjecture for some pair of representations of $U(3)\times U(2)$ involving a non-generic representation. For a pair of generic $L$-parameters of $(U(n),U(n-1))$, it is known that there is a unique pair of representations in their associated Vogan $L$-packets which produces the unique Bessel model of these $L$-parameters. We showed that this is not ture for some pair of $L$-parameters involving a non-generic one.\\ On the other hand, we give the precise local theta correspondence for $(U(1),U(3))$ not at the level of $L$-parameters but of individual representations in the framework of the local Langlands correspondence for unitary group. As an applicaiton of these results, we prove an analog of Ichino-Ikeda conejcture for some non-tempered case. 
The main tools in this work are the see-saw identity, local theta correspondence for (almost) equal rank cases and recent results on the local Gan-Gross-Prasad conjecture both on the Fourier-Jacobi and the Bessel case.

\end{abstract}
\maketitle
\section{\textbf{Introduction}}Let $E/F$ be a quadratic extension of number fields. Let $G=U(3)$ be the unitary group relative to $E/F$. In \cite{Ro}, Rogawski has defined a certain enlarged class of $L$-packet, or $A$-packet, using the endoscopic transfer of a one-dimensional character of the $H=U(2) \times U(1)$, which is the unique elliptic endoscopic group for $G$. 

More precisely, let $\varrho= \otimes_v \varrho_v$ be a one-dimensional automorphic character of $H$. An $A$-packet $\Pi(\varrho) \simeq \otimes \Pi(\varrho_v)$ is the transfer of $\varrho$ with respect to functoriality for an embedding of $L$-groups $\xi : ^LH \to ^LG$. Then for all place $v$ of $F$, $\Pi(\varrho_v)$ contains a certain non-tempered representation $\pi^{n}(\varrho_v)$ and for a place which remains prime in $E$, it has an additional supercuspidal representation $\pi^{s}(\varrho_v)$. Gelbart and Rogawski \cite{Ge1} showed that the representations in this $A$-packet arise in the Weil representation of $G$. The main purpose of this article is to study the branching of the representations in this $A$-packet when restricted to $U(2)$ over local fields.

For the branching problem of codimension 1 classical groups, there is the so called \emph{Gan-Gross-Prasad} (GGP) conjecture, which was first formulated by Gross and Prasad for orthogonal groups and later, together with Gan, they extended it to all classical groups in \cite{Gan2}. But this paper mainly concerns unitary groups, we shall state the GGP conjecture only for unitary groups.

Let $E/F$ be a quadratic extension of local fields of characteristic zero. Let $V_{n+1}$ be a Hermitian space of dimension $n+1$ over $E$ and $W_n$ a skew-Hermitian space of dimension $n$ over $E$.
Let $V_n \subset V_{n+1}$  be a nondegenerate subspace of codimension $1$,  so that if we set 
\[  G_n =  \U(V_n) \times \U(V_{n+1}) \quad \text{or} \quad \U(W_n) \times \U(W_n) \]
and
\[   H_n = \U(V_n) \quad \text{or} \quad \U(W_n), \]
then we have a diagonal embedding
\[ \Delta:  H_n \hookrightarrow G_n. \]

Let $\pi$ be an irreducible smooth representation of $G_n$. In the Hermitian case, one is interested in computing 
\[  \dim_\CC \Hom_{\Delta H_n} ( \pi, \CC). \]
We shall call this the \emph{Bessel} case (B) of the GGP conjecture.
To describe the GGP conjecture for the skew-Hermitian case, we need another piece of data: a Weil representation $\omega_{\psi, \chi, W_n}$, where $\psi$ is a nontrivial additive character of $F$ and $\chi$ is a character of $E^{\times}$ whose restriction to $F^{\times}$ is the quadratic character $\omega_{E/F}$ associated to $E/F$ by local class field theory.
Then  one is interested in computing
\[  \dim_\CC \Hom_{\Delta H_n} ( \pi, \omega_{\psi,\chi, W_n}). \]
We shall call this the \emph{Fourier--Jacobi} case (FJ) of the GGP conjecture. To unify notation, we shall let $\nu = \CC$ or $\omega_{\psi,\chi, W_n}$ in the respective cases.

By the result of \cite{agrs}, it is known \[  \dim_\CC \Hom_{\Delta H_n} ( \pi, \nu ) \le 1 \] and so we want to specify irreducible smooth representations $\pi$ such that $$\Hom_{\Delta H_n} ( \pi, \nu ) = 1. $$

In \cite{Gan2}, Gan, Gross and Prasad has brought this problem into a more general setting using the notion of relevant pure inner forms of $G_n$ and Vogan $L$-packets. A pure inner form of $G_n$ is a group of the form 
\[  G_n' = \U(V_{n+1}') \times \U(V'_{n}) \quad \text{or} \quad \U(W'_n) \times \U(W''_n) \]
where $V_{n}' \subset V_{n+1}'$ are $n$ and $n+1$ dimensional hermitian spaces over $E$ and $W_n'$ is a $n$-dimensional skew-hermitian spaces over $E$.

\noindent Furthermore, if \[  \quad  V_{n+1}'/V_{n}' \cong V_{n+1}/V_{n} \quad \text{or} \quad W_n'=W_n'', \] we say that $G_n'$ is a relevant pure inner form of $G_n$.\\ (Indeed, there are four pure inner forms of $G_n$ and among them, only two are relevant.)

If $G_n'$ is relevant, we set
\[   H'_n = \U(V'_n) \quad \text{or} \quad \U(W'_n) \]
so that we have a diagonal embedding
\[ \Delta:  H'_n \hookrightarrow G'_n. \]

Suppose that $\phi$ is an $L$-parameter for the group $G_n$. Then the (relevant) Vogan $L$-packet $\Pi_{\phi}$ associated to $\phi$  consists of certain irreducible smooth representations of $G_n$ and its (relevant) pure inner forms $G_n'$ whose $L$-parameter is $\phi$. We denote the relevant Vogan $L$-packet of $\phi$ by $\Pi^R_{\phi}$.

Now we can loosely state the GGP conjecture  as follows:

\begin{GPconj}
For a generic $L$-parameter $\phi$ of $G_n$, the followings hold:
\begin{enumerate}
\item $\sum_{\pi' \in \Pi^R_{\phi}}\dim_{\CC}\Hom_{\Delta H'_n} ( \pi', \nu )=1.$

\item Using the local Langlands correspondence for unitary group, we can pinpoint $\pi' \in \Pi^R_{\phi}$ such that $$\dim_{\CC}\Hom_{\Delta H'_n} ( \pi', \nu )=1.$$ 
\end{enumerate}
\end{GPconj}

For tempered $L$-parameter $\phi$, Beuzart-Plessis \cite{bp1},\cite{bp2},\cite{bp3} established (B) of the GGP conjecture. Building upon Plessis's work, Gan and Ichino \cite{iw} proved (FJ) for tempered case first by establishing the precise local theta correspondence for almost equal rank unitray groups and then extended both (B) and (FJ) to generic cases. We shall elaborate more on this in \S \ref{gen}.

In this paper, we shall investigate conjecture (B) for $G_2$ when an $L$-parameter of $G_2$ involves some non-generic $L$-parameter of $U(V_3)$. More precisely, we have

\begin{main}Denote the $L$-parameter of the non-generic representation $\pi^{n}(\varrho_v)$ appearing in the non-tempered $A$-packet $\Pi(\varrho_v)$ by $\phi^{n}$. Let $\phi^t$ be a tempered $L$-parameter of $U(V_2)$ and $\phi=\phi^n \otimes \phi^t$ the $L$-parameter of $G_2=U(V_3) \times U(V_2)$.
Then

\begin{enumerate}
\item If the $L$-parameter $\phi^t$ does not come from the theta lift of $U(W_1)$, $$\sum_{\pi' \in \Pi^R_{\phi}}\dim_{\CC}\Hom_{\Delta H'_2} ( \pi', \CC )=0$$
\item Suppose that $\phi^t$ comes from the theta lift of $U(W_1)$. Then $$\sum_{\pi' \in \Pi^R_{\phi}}\dim_{\CC}\Hom_{\Delta H'_2} ( \pi', \CC )=1$$ and using the local Langlands correspondence, we can explicitly describe $\pi' \in \Pi^R_{\phi}$ such that \begin{equation}\label{p}\dim_{\CC}\Hom_{\Delta H'_2} ( \pi', \CC )=1.\end{equation}
\end{enumerate}

\end{main}

The main idea for this is to consider the following see-saw diagram:
\[
 \xymatrix{
  \U(W_1)  \times \U(W_1)  \ar@{-}[dr] \ar@{-}[d] & \U(V_{3}) \ar@{-}[d] \\
  \U(W_1) \ar@{-}[ur] &  \U(V_2) \times \U(V_{1})}.
\]

Write $\pi=\pi_3 \otimes \pi_2$ in $\Pi^R_{\phi}$ where $\pi_3\in \Pi_{\phi^n}$ and $\pi_2 \in \Pi_{\phi^t}$.
Since all elements in the $L$-packet $\Pi_{\phi^n}$ can be obtained by the theta lift from $U(W_1)$, we can write $\pi_3=\Theta_{\psi,\chi,V_3,W_1}(\sigma)$ where $\sigma$ is an irreducible smooth character of $U(W_1)$ and $\psi,\chi$ are some characters needed to fix a relevant Weil representation. By the see-saw identity, we have 
$$\Hom_{U(V_2)}(\Theta_{\psi,\chi,V_3,W_1}(\sigma)\otimes \pi_2, \mathbb{C})  \simeq \Hom_{U(V_2)}(\Theta_{\psi,\chi,V_3,W_1}(\sigma),\pi^{\vee}_2) \simeq \Hom_{U(W_1)}(\Theta_{\psi,\chi,W_1,V_2}(\pi^{\vee}_2) \otimes \omega_{\psi,\chi,W_1,V_1},\sigma)$$ where $\pi^{\vee}_2$ is the contragredient representation of $\pi_2$. 

\noindent From this, we see that for having $\Hom_{U(V_2)}(\Theta_{\psi,\chi,V_3,W_1}(\sigma)\otimes \pi_2, \mathbb{C})\ne 0$, we must have $\Theta_{\psi,\chi,W_1,V_2}(\pi^{\vee}_2)\ne 0$, i.e., $\pi^{\vee}_2$ should be the theta lift from $U(W_1)$. This accounts for (i) in the \textbf{Main Theorem}. 

\begin{rem}\label{rem1.1}As we shall see later, the $L$-parameter $\phi^n$ is not only non-tempered but also non-generic. Thus (i) says that the GGP conjecture without the genericity hypothesis is not true.
\end{rem}

If $\Theta_{\psi,\chi,W_1,V_2}(\pi^{\vee}_2)\ne 0$, then we can apply (FJ) for $U(W_1)$ to find the precise representations in the Vogan $L$-packet associated to the $L$-parameters of $\Theta_{\psi,\chi,W_1,V_2}(\pi^{\vee}_2)$ and $\sigma$. However, to find a representation $\pi'$ in (\ref{p}), we need to know the precise local theta correspondence between $(U(W_1),U(V_3))$ as well as $(U(W_1),U(V_2))$. For the precise local theta correspondence for codimension 0 and 1 cases, it was suggested and proved by Gan and Ichino in \cite{iw} and that of $(U(W_1),U(V_2))$ immediately follows from it. In Theorem \ref{thm1}, we suggest and prove the local theta correspondence for $(U(W_1),U(V_3))$ as our second main result, which is not investigated so far in general.\begin{rem}For orthogonal group, such restriction problem for non-tempered $L$-parameter was already dealt in \cite{gg} and \cite{na}. Thus this paper can be seen as an analog of their works in the unitary group.
\end{rem}

As an applicaition of our \textbf{Main Theorem}, we could establish the Ichino-Ikeda conjecture of the unitary group for the non-tempered case. In \cite{Ich}, Ichino and Ikeda defined the local period for a pair of tempered representations of orthogonal group using the matrix coefficients and conjectured that its non-vanishing would be equivalent to the existence of its Bessel model. This conjecture was settled by Waldspuger for the non-archimedean case, and Sakellaridis, Venkatesh \cite[\S 6.4]{SV} and Plessis \cite[\S 14.3]{bp2} established this conjecture in the setting of unitary group independently. However, if one considers a pair involving a non-tempered representation, the local period may diverge and so one needs to regularise it. In \cite{Haan}, the author introduced the regularized local period $\mathcal{P}_v^{reg}$ for some special pair of non-tempered representations. With this notion of regularised local period, we could prove \begin{equation}\label{ii}\Hom_{U(V_{2,v})}(\pi_{3,v},\pi_{2,v})\ne0 \Leftrightarrow \mathcal{P}^{reg}_v \ne 0 \quad  \text{for all non-archimedean places} \ v \text{ of } F\end{equation} when
$ \pi_{3,v}=\Theta_{\psi,\chi,V_{3,v},W_{1,v}}(\overline{\sigma_v})$, $\pi_{2,v}=\Theta_{\psi,\chi,V_{3,v},W_{1,v}}(\overline{\mathbb{I}_v})$ are local theta lifts for some irreducible and trivial representations $\overline{\sigma_v}, \overline{\mathbb{I}_v}$ of $U(W_{1,v})$ respectively. (For a complex representation $\pi$ of a group $G$, we denote its complex conjugate representation by $\overline{\pi}$.) This justifies the definition of $\mathcal{P}_v^{reg}$ for regularised local period and as a corollary, we can express the main result of \cite{Haan} in the form of the original global Gross-Prasad conjecture under the assumption that (\ref{ii}) also holds for archimedean places.

\begin{thm}Let $\pi_{3}=\Theta_{\psi,\chi,V_{3},W_{1}}(\overline{\sigma})$ and $\pi_{2}=\Theta_{\psi,\chi,V_{3},W_{1}}(\overline{\mathbb{I}})$ be the global theta lifts of some automorphic character $\sigma$ and trivial chracter $\mathbb{I}$ of $U(W_{1,v})$ to $U(V_{3,v})$ and $U(V_{2,v})$, respectively.\\Assume (\ref{ii}) holds for archimedean cases. \\
If $\Hom_{U(V_{2,v})}(\pi_{3,v},\pi_{2,v})\ne 0$ for all places $v$ of $F$, then we have $$
\mathcal{P}\ne0 \Leftrightarrow L_E(\frac{1}{2},BC( \sigma)\otimes \chi)\ne 0$$ where $\mathcal{P}$ is the global period functional of $\pi_3 \times \pi_2$ defined by $$\mathcal{P}(f_{\pi_3},f_{\pi_2})=\int_{U(V_2)(F)\backslash U(V_2)(\mathbb{A}_F)}f_{\pi_3}(g) \overline{f_{\pi_2}(g)}dg \quad \text{ for } f_{\pi_3}\in \pi_3 \text{ and } f_{\pi_2}\in \pi_2.$$ 

\end{thm}

The rest of the paper is organized as follows; In Section 2, we give a brief sketch of the local Langlands correspondence for unitary group. In Section 3, we collect some results of the local theta correspondence for unitary group which we shall use in the proof of our main results. In Section 4, we prove Theorem \ref{thm2} under the assumption of Theorem \ref{thm1} whose proof will occupy Section 5. Finally, we give an application of our main result in Section 6. 

\subsection*{Acknowledgements} The author expresses deep gratitude to his professor Haseo Ki, Dongho Byeon for their guidance, patience and constant encouragement during his whole graduate years. Without their huge influence on me, this paper would never come out this world. 
\\ Originally, this paper grew out as an attempt to understand Gan and Ichino's work \cite{iw}. Besides its large inspiration on this work, we are much indebted to professor Atsushi Ichino for his suggestion this problem and to professor Wee Teck Gan for helpful discussion at 2014 ICM Seoul. Taking this opportunity, we would like to express our sincere thanks to them. Part of this work was completed during the stay at KAIST and so we are grateful for their hospitality and providing a wonderful place to conduct a research. Finally, the author would like to thank the referee for his careful reading on the manuscript and numerous corrections and helpful suggestions.
\subsection{Notation}\label{not}We fix some notations we shall use throughout this paper:
\begin{itemize}
\item  $E/F$ is a quadratic extension of number fields or local fields of characteristic zero.

\item  $c$ is the non-trivial element of Gal$(E/F)$.

\item $\text{Fr}_E$ is a Frobenius element of Gal$(\bar{E}/E)$.

\item  Denote by $\Tr_{E/F}$ and $\N_{E/F}$ the trace and norm maps from $E$ to $F$.

\item $E^1:=\{x\in E\ | \ \N_{E/F}(x)=1\}$ 

\item $\delta$ is an element of $E^{\times}$ such that  $\Tr_{E/F}(\delta)=0$.

\item Let $\psi$ be an additive character of $\mathbb{A}_F/F$ or $F$ and define $$\psi^E(x)  := \psi(\frac{1}{2}\Tr_{E/F}(\delta x)) \quad \text{and} \quad \psi^E_2(x)  := \psi(\Tr_{E/F}(\delta x)).$$

\item  Let $\omega_{E/F}$ be the quadratic character assosiated to $E/F$ by global or local class field theory and let $\chi$ be a character of $\mathbb{A}_E^{\times}/E^{\times}$ or $E^{\times}$ whose resriction to $\mathbb{A}_F^{\times}/F^{\times}$ of $F^{\times}$ is $\omega_{E/F}$.

\item For an linear algebraic group $G$, denote its $F$-points by $G(F)$.

\item $\mathbb{I}_G$ denotes the trivial representation of $G$.
\end{itemize}

\section{\textbf{Local Langlands correspondence for unitary group}}
By the recent work of Mok \cite{Mok}, and Kaletha-M\'inguez-Shin-White \cite{kmsw}, the local Langlands correpondence is now known for unitary group conditional on the stabilization of the twisted trace formula. Since our main results are expressed using the local Langlands correspondence, we shall assume the local Langlands correspondence for unitary group throughout the paper. In this section, we list some of its properties which we shall use later. Indeed, much part of this section are excerpts from Section 2 in \cite{iw}.

\subsection{Hermitian and skew-Hermitian spaces}
Until Section 5, we will use the symbol $E/F$ to denote the quadratic extension of local fields of characteristic zero.
For $\varepsilon \in \{\pm1\}$, let $V$ be a finite $n$-dimensional vector space over $E$ equipped with a nondegenerate $\varepsilon$-Hermitian $c$-sesquilinear form $\langle \cdot, \cdot \rangle_V : V \times V \rightarrow E$.
That means for $v, w \in V$ and $a, b \in E$, we have \[
 \langle a v, b w \rangle_V = a b^c \langle v, w \rangle_V, \qquad
 \langle w, v \rangle_V = \varepsilon \cdot \langle v, w \rangle_V^c.
\]

\noindent Denote $\disc V = (-1)^{(n-1)n/2} \cdot \det V$, so that 
\[
 \disc V \in
 \begin{cases}
  F^{\times} / \mathrm{N}_{E/F}(E^{\times})
  & \text{if $\varepsilon = +1$;} \\
  \delta^n \cdot F^{\times} / \mathrm{N}_{E/F}(E^{\times})
  & \text{if $\varepsilon = -1$.}
 \end{cases}
\]
Then we can define $\epsilon(V) = \pm 1$ by 
\begin{equation}\label{sign}
 \epsilon(V) = 
 \begin{cases}
  \omega_{E/F}(\disc V) & \text{if $\varepsilon = +1$;} \\
  \omega_{E/F}(\delta^{-n} \cdot \disc V) & \text{if $\varepsilon = -1$.}
 \end{cases}
\end{equation}
By a theorem of Landherr, for a given positive integer $n$, there are exactly two isomorphism classes of hermitian spaces of dimension $n$ distinguished by $\epsilon(V)$.
Let $\U(V)$ be the unitary group of $V$ defined by
\[
  \U(V) = \{ g \in \GL(V) \, | \,
 \text{$\langle g v, g w \rangle_V =  \langle v, w \rangle_V$ for $v, w \in V$}
 \}.
\]
Then $U(V)$ turns out to be a connected reductive algebraic group defined over $F$.
\subsection{$L$-parameters and component groups}

Let $I_F$ be the inertia subgroup of $\text{Gal}(\bar{F}/F)$. Let $W_F=I_F \ltimes \langle \text{Fr}_F \rangle $ be the Weil group of $F$ and $\WD_F = W_F \times \SL_2(\CC)$ the Weil-Deligne group of $F$.
For a homomorphism $\phi: \WD_F \rightarrow \GL_n(\CC)$, we say that it is a representation of $\WD_F$ if 
\begin{enumerate}
\item $\phi$ is trivial on an open subgroup of $I_F$,
 \item $\phi$ is continuous and $\phi(\text{Fr}_F)$ is semisimple, 
 \item the restriction of $\phi$ to $\SL_2(\CC)$ is induced by a morphism of algebraic groups $\SL_2 \to \GL_n$
\end{enumerate}
For a representation $\phi$ of $\WD_F$, we say that $\phi$ is tempered when the image of $W_F$ is bounded.\\
Define $\phi^{\vee}$ by $\phi^\vee(w) = {}^t\phi(w)^{-1}$ and call this the contragredient representation of $\phi$. For $s \in W_F \smallsetminus W_E$, we define a representation $\phi_s^c$ of $\WD_E$ by $\phi_s^c(w) = \phi(sws^{-1})$.
It is known that the equivalence class of $\phi_s^c$ is independent of the choice of $s$. As we are only interested in the equivalence classes of representations, we shall denote $\phi_s^c$ by $\phi^c$ suppressing $s$. Then we say that $\phi$ is conjugate self-dual if there is an isomorphism $b: \phi \mapsto (\phi^{\vee})^{c}$ and for $\varepsilon=\pm 1$, we say that $\phi$ is conjugate self-dual with sign $\varepsilon$ if $(b^{\vee})^c=\varepsilon \cdot b$.

Let $V$ be an $n$-dimensional $\varepsilon$-hermitian space over $E$. An $L$-parameter for the unitary group $\U(V)$ is a homomorphism $$\phi:\WD_F \longrightarrow ^LU(V)=GL_{n}(\CC) \rtimes \text{Gal}(E/F)$$ such that 
\begin{itemize}
\item the composite of $\phi$ with the projection onto $GL_{n}(\CC)$ is a representation of $WD_F$
\item the composite of $\phi$ with the projection onto Gal$(E/F)$ is the natural projection of $\WD_F$ to Gal$(E/F)$.

\end{itemize}

The following proposition in \cite[\S 8]{Gan2} enables us to remove the cumbersome Gal$(E/F)$-factor in the definition of $L$-parameter of $U(V)$.

\begin{prop}Restriction to $W_E$ of $W_F$ in $\WD_F$ gives a bijection between the set of $L$-parameters for $U(V)$ and the set of equivalence classes of conjugate self-dual representations $$\phi:\WD_E \longrightarrow GL_n(\CC)$$ 
of sign $(-1)^{n-1}$.
\end{prop}
With this proposition, henceforth, we shall mean an $L$-parameter for $\U(V)$ by $n$-dimensional conjugate self-dual representation $\phi$ of $\WD_E$ with sign $(-1)^{n-1}$.

Given an $L$-parameter $\phi$ of $U(V)$, we can write $\phi$ as a direct sum
\[  \phi = \bigoplus_i m_i \phi_i \]
with pairwise inequivalent irreducible representations $\phi_i$ of $\WD_E$ with multiplicities $m_i$.
We say that $\phi$ is square-integrable if for all $i$, $m_i=1$ and $\phi_i$ is conjugate self-dual with sign $(-1)^{n-1}$.

Given an $L$-parameter $\phi$ for $\U(V)$, we can associate its component group $S_{\phi}$.
As explained in \cite[\S 8]{Gan2}, $S_{\phi}$ is a finite 2-abelian group and it has a form
\[  S_{\phi}  = \prod_j  (\ZZ / 2\ZZ) a_j \]
with a canonical basis $\{ a_j \}$, where the product ranges over all $j$ such that $\phi_j$ is conjugate self-dual with sign $(-1)^{n-1}$.
If we denote the image of $-1 \in \GL_n(\CC)$ in $S_\phi$ by $z_\phi$, it is known that\[
 z_{\phi} = (m_j a_j) \in \prod_j  (\ZZ / 2\ZZ) a_j.
\]

\subsection{Local Langlands correspondence for unitary group}

The aim of the local Langlands correspondence for unitary groups is to classify the irreducible smooth representations of unitary groups. To state it, we first introduce some notations.
\begin{itemize}
\item Let $V^+$ and $V^-$ be the $n$-dimensional $\varepsilon$-Hermitian spaces with $\epsilon(V^+) = +1$, $\epsilon(V^-) = -1$ respectively.
\item For an $L$-parameter $\phi$ of $U(V^{\pm})$, let $\Pi_{\phi}$ be the Vogan $L$-packet associated to $\phi$, which is a finite set of irreducible smooth representations of $U(V^{\pm})$.
\item Let $\Irr(\U(V^{\pm}))$ be the set of irreducible smooth representations of $\U(V^{\pm})$.

\end{itemize}

Then the local Langlands correspondence in a form suggested by Vogan \cite{v}, says that there is one-to-one correspondence between  
\[  \Irr(\U(V^+)) \sqcup \Irr(\U(V^-))   \longleftrightarrow  \bigsqcup_{\phi}  \Pi_{\phi}, \]
where $\phi$ on the right-hand side runs over all equivalence classes of $L$-parameters for $\U(V^\pm)$.

\noindent Then under the local Langlands correspondence, we may also decompose $\Pi_{\phi}$ as 
\[  \Pi_{\phi}  = \Pi_{\phi}^+\sqcup   \Pi_{\phi}^-, \]
where for $\epsilon = \pm 1$, $\Pi_{\phi}^{\epsilon}$ consists of the representations of $\U(V^{\epsilon})$ in $\Pi_{\phi}$.

As explained in \cite[\S 12]{Gan2}, once an additive character $\psi:F \to \CC$ is chosen,  there is an associated bijection $$J^{\psi}(\phi):\Pi_{\phi} \to \Irr( S_{\phi}).$$ When $n$ is odd, this bijection is always canonical and so does not depend on the choice of $\psi$. However, when $n$ is even, it depends on the choice of an additive character of $\psi:F \to \CC^{\times}$. More precisely, such bijection is determined by the $\N_{E/F}(E^{\times})$-orbit of nontrivial additive characters 
\[
\begin{cases}
 \psi^E:E/F \rightarrow \CC^{\times} & \text{if $\varepsilon = +1$;} \\
 \psi:F \rightarrow \CC^{\times} & \text{if $\varepsilon = -1$}
\end{cases}
\] 
where \[  \psi^E(x) : = \psi(\tfrac{1}{2} \Tr_{E/F}(\delta x)) \]
Hereafter, if a nontrivial additive character $\psi:F \rightarrow \CC^{\times}$ is fixed, we use the notation $$J^{\psi}(\phi):\Pi_{\phi} \to \Irr( S_{\phi})$$ as above once and for all. 

The map $J^{\psi}$ enables us to label all irreducible smooth representations of $\U(V^{\pm})$ with $\pi(\phi,\eta)$ for some unique $L$-parameter $\phi$ of $U(V^{\pm})$ and $\eta \in \Irr( S_{\phi})$. Once an additive character $\psi:F \to \CC$ is fixed, we shall represent all irreducible smooth representations of $\U(V^{\pm})$ as $\pi(\phi,\eta)$ with the implicit use of $J^{\psi}$.

\subsection{Properties of the local Langlands correspondence}
\label{SS:llc}

We briefly list some properties of the local Langlands correspondence for unitary group, which we will use in this paper:

\begin{itemize}
\item
$\pi(\phi,\eta)$ is a representation of $\U(V^{\epsilon})$ if and only if $\eta(z_\phi)  = \epsilon$.

\item $\pi(\phi,\eta)$ is tempered if and only if $\phi$ is tempered.

\item $\pi(\phi,\eta)$ is square-integrable if and only if $\phi$ is square-integrable.

\item The component groups $S_{\phi}$ and $S_{\phi^{\vee}}$ are canonically identified. Under this canonical identification, if we write $\pi$ as $\pi(\phi,\eta)$, then the contragradient representation $\pi^{\vee}$ of $\pi$, has $L$-parameter $\phi^{\vee}$ and associated character $\eta^{\vee}=\eta \cdot \nu$ where \[  
\nu(a_j)  = \begin{cases}
\omega_{E/F}(-1)^{\dim \phi_j} & \text{if $\dim_{\CC} \phi$ is even;}  \\
1 & \text{if $\dim_{\CC} \phi$ is odd.} \end{cases} \]\\
This property follows from a result of Kaletha \cite[\S 4]{kel} for the case of unitary groups.
\item If $\phi$ is a non-tempered $L$-parameter, we can decompose 
\[
 \phi = \phi_1 \oplus \cdots \oplus \phi_r \oplus \phi_0
 \oplus (\phi_r^c)^\vee \oplus \cdots \oplus (\phi_1^c)^\vee,
\]
where
\begin{itemize}
 \item for $1\le i \le r$, $\phi_i$ is a $k_i$-dimensional representation of $\WD_E$ of the form $\phi_i = \phi_i' \otimes |\cdot|^{e_i}$ for some tempered representation $\phi_i'$ of $\WD_E$ and some real number $e_i$ such that
\[
 e_1 > \cdots > e_r > 0,
\]
 \item $\phi_0$ is a tempered $L$-parameter for $\U(V_0^{\pm})$,
 where $V_0^{\pm}$ are the $\varepsilon$-Hermitian spaces of dimension $n-2(k_1+\cdots+k_r)$ over $E$.
\end{itemize}
We note that the natural map $S_{\phi_0} \rightarrow S_\phi$ is an isomorphism.

\end{itemize}

\section{\textbf{Local theta correspondence}}

In this section, we state the local theta correspondence for three pairs of unitary groups, namely, $(U(1),U(1))$, $(U(1),U(2))$, $(U(1),U(3))$. From now on, for $\epsilon = \pm 1$, we shall denote by $V_n^\epsilon$ the $n$-dimensional Hermitian space with $\epsilon(V_n^\epsilon) = \epsilon$ and by $W_n^\epsilon$ the $n$-dimensional skew-Hermitian space with $\epsilon(W_n^\epsilon) = \epsilon$, so that $W_n^\epsilon = \delta \cdot V_n^\epsilon$.

\subsection{The Weil representation for Unitary groups}
\label{SS:Weil}In this subsection, we introduce the Weil representation. Since the constructions of global and local Weil representation are similar, we will treat both of them simultaneously. If the same statement can be applied to both the local and global cases, we will not use the distinguished notation $U(V)(F)$ and $U(V)(\mathbb{A}_F)$, but just refer them to $U(V)$.

Let $E/F$ be a quadratic extension of local or global fields and $(V,\langle,\rangle_{V})$ be a Hermitian space and $(W,\langle,\rangle_{W})$ a skew-Hermitian space over $E$.

\noindent Define the symplectic space
$$
\mathbb{W}_{V,W} := \operatorname{Res}_{E/F} V \otimes_E W
$$
with the symplectic form $$
\langle v \otimes w,v' \otimes w' \rangle_{\mathbb{W}_{V,W}} := \frac{1}{2}\operatorname{tr}_{E/F}\left(\langle v,v'\rangle_{V}\otimes {\langle w,w' \rangle}_{W}\right).
$$
We also consider the associated symplectic group $Sp(\mathbb{W}_{V,W})$ preserving $\langle \cdot,\cdot \rangle_{\mathbb{W}_{V,W}}$ and the metaplectic group $\widetilde{Sp}(\mathbb{W}_{V,W})$ satisfying the following short exact sequence :

$$1\to \CC^\times\to\widetilde{Sp}(\mathbb{W}_{V,W})\to Sp(\mathbb{W}_{V,W})\to 1.$$\\ 
Let $\mathbb{X}_{V,W}$ be a Lagrangian subspace of $\mathbb{W}_{V,W}$ and we fix an additive character $\psi:\mathbb{A}_F/F\to\CC^\times$ (globally) or $\psi:F\to\CC^\times$ (locally). Then we have a Schr\"odinger model of the Weil Representation $\omega_\psi$ of $\widetilde{Sp}(\mathbb{W})$ on $\mathcal{S}(\mathbb{X}_{V,W})$, where $\mathcal{S}$ is the Schwartz-Bruhat function space.

Once and for all, we fix an unitary character $\chi$ of $\mathbb{A}_E^{\times}/E^{\times}$ or $E^{\times}$ whose restriction to $\mathbb{A}_F^{\times}/F^{\times}$ of $F^{\times}$ is $\omega_{E/F}$. Let $\chi_V,\chi_W$ be unitary characters of  $\mathbb{A}_E^{\times}/E^{\times}$ or $E^{\times}$ such that 
$$\chi_V|_{\mathbb{A}_F^{\times}/F^{\times} \text{ or } F^{\times}}:=\omega_{E/F}^{\text{dim}_E W} \quad \text{ and } \quad \chi_W|_{\mathbb{A}_F^{\times}/F^{\times} \text{ or } F^{\times}}:=\omega_{E/F}^{\text{dim}_E V}.$$

\noindent By \cite[\S 1]{Ha}, such a choice $(\chi_{V}, \chi_{W})$ determines a splitting homomorphism $$
\iota_{\chi_{V}, \chi_{W}}:U(V)\times U(W)\to \widetilde{Sp}(\mathbb{W}_{V,W})$$ and so by composing this to $\omega_\psi$, we have a Weil representation $\omega_\psi\circ \iota_{\chi_{V},\chi_{W}}$ of $U(V) \times U(W)$ on $\mathbb{S}(\mathbb{X}_{V,W})$.

Throughout the rest of the paper, when it comes to a Weil representations of $U(V) \times U(W)$, we shall denote $ \omega_\psi\circ \iota_{\chi_{V},\chi_{W}}$ by $\omega_{\psi,W,V}$ with understanding the choices of characters $(\chi_V,\chi_W)$ were made as follows: 
$$\chi_V=\chi^{\text{dim}_E W} \quad \text{ and } \quad \chi_W=\chi^{\text{dim}_E V}.$$

\begin{rem} When $\dim_E W=1$, the image of $U(W)$ in $\widetilde{Sp}(\mathbb{W}_{V,W})$ coincides with the image of the center of $U(V)$, so we can regard the Weil representation of $U(V) \times U(W)$  as the representation of $U(V)$ and we denote the Weil representation as just $\omega_{\psi,V}$.
\end{rem}

\subsection{Local theta correspondence}

Given a Weil representation $\omega_{\psi,W,V}$ of $U(V) \times U(W)$ and an irreducible smooth representation $\pi$ of $U(W)$, the maximal $\pi$-isotypic quotient of $\omega_{\psi,V,W}$, say $\mathbb{S}(\pi)$, is of the form
\[\mathbb{S}(\pi) \cong \Theta_{\psi,V,W}(\pi) \boxtimes \pi
\]
for some smooth representation $\Theta_{\psi,V,W}(\pi)$ of $U(V)$ of finite length.
By the Howe duality, which was first proved by Waldspurger \cite{w} except for $p\ne 2$ and recently completed by Gan and Takeda in \cite{gt1}, \cite{gt2}, the maximal semisimple quotient $\theta_{\psi,V,W}(\pi)$ of $\Theta_{\psi,V,W}(\pi)$ is either zero or irreducible.

In this paper, we consider three kinds of theta correspondences for $\U(V) \times \U(W)$ : 
\begin{enumerate}
\item $\dim V =1 \ \text{,} \ \dim W=1$

\item $\dim V =2 \ \text{,} \ \dim W=1$

\item $\dim V =3 \ \text{,} \ \dim W=1$
\end{enumerate}

For the cases $|\dim V - \dim W|=0 \text{ and } 1 $, D.~Prasad \cite{pra} conjectured the local theta correspondence in terms of the local Langlands correspondence and quite recently, Gan and Ichino proved both cases in \cite{gi}, \cite{iw}. For what follows, we fix an additive character $\psi:F \to \CC^{\times}$.

\subsection{Case (i)} 

We first consider the theta correspondence for $\U(V_1^\epsilon) \times \U(W_1^{\epsilon'})$.
The following summarizes some results of \cite{Gan3}, \cite{gi}, \cite{iw}, which are specialized to this case.

\begin{thm}  \label{Te}
Let $\phi$ be an $L$-parameter for $\U(W_1^\pm)$.
Then we have:
\begin{enumerate}
\item For any fixed $\pi \in \Pi_{\phi}^{\epsilon'}$, exactly one of $\Theta_{\psi, V_1^+,W_1^{\epsilon'}}(\pi)$ or $\Theta_{\psi, V_1^-,W_1^{\epsilon'}}(\pi)$ is nonzero.

\item $\Theta_{\psi, V_1^\epsilon, W_1^{\epsilon'}}(\pi) \ne 0$ if and only if
\[  \epsilon(\tfrac{1}{2}, \phi  \otimes  \chi^{-1}, \psi^E_2)  = \epsilon \cdot \epsilon', \]
where
\[  \psi^E_2(x)  = \psi(\Tr_{E/F}(\delta x)).  \]

\item If $\Theta_{\psi, V_1^{\epsilon},W_1^{\epsilon'}}(\pi)$ is nonzero, then $\theta_{\psi, V_1^{\epsilon},W_1^{\epsilon'}}(\pi)$ has $L$-parameter
\[   \theta(\phi)  =  \phi . \]

\item The theta correspondence $\pi \mapsto   \theta_{\psi, V_1^{\epsilon}, W_1^{\epsilon'}}(\pi)$ gives a bijection
\[  \Pi_{\phi}  \longleftrightarrow   \Pi_{\theta(\phi)}. \]

\item Let $S_{\phi}=S_{\theta(\phi)}=(\ZZ/2\ZZ)a_1.$ 
Since $n=1$ and $\phi=\theta(\phi)$, there is the same bijection \[  J^{\psi}(\phi):  \Pi_{\phi} \longleftrightarrow \Irr(S_{\phi})   \quad \text{and} \quad J^{\psi}(\theta(\phi)):  \Pi_{\theta(\phi)} \longleftrightarrow \Irr(S_{\theta(\phi)}). \]
With these bijections, we can describe the bijection
\begin{align*}
 \Irr(S_{\phi}) & \longleftrightarrow \Irr(S_{\theta(\phi)}) \\
 \eta & \longleftrightarrow \theta(\eta)
\end{align*}
induced by the theta correspondence in (iv) as follows:
\[  \theta(\eta)(a_1) =  \eta(a_1)  \cdot \epsilon(\tfrac{1}{2}, \phi \otimes  \chi^{-1},  \psi^E_2). \]

\item If $\Theta_{\psi, V_1^{\epsilon}, W_1^{\epsilon'}}(\pi)$ is nonzero, then $\Theta_{\psi, V_1^{\epsilon}, W_1^{\epsilon'}}(\pi)$ is irreducible and so $\Theta_{\psi, V_1^{\epsilon}, W_1^{\epsilon'}}(\pi)=\theta_{\psi, V_1^{\epsilon}, W_1^{\epsilon'}}(\pi)$.

\end{enumerate}
\end{thm}

\subsection{Case (ii)}

Now we shall consider the theta correspondence for $\U(V_{2}^\epsilon) \times \U(W_1^{\epsilon'})$. The following summarizes some results of \cite{Gan3}, \cite{gi}, \cite{iw}, which are specialized to this case.

\begin{thm}  \label{Tae}

Let $\phi$ be an $L$-parameter for $\U(W_1^{\pm})$. Then we have:

\begin{enumerate}

\item Suppose that $\phi=\chi^2$.

\begin{enumerate}

 \item For any $\pi \in \Pi_{\phi}^{\epsilon'}$, 
 $\Theta_{\psi, V_{2}^{\epsilon}, W_1^{\epsilon'}}(\pi)$ is nonzero and $\theta_{\psi, V_{2}^{\epsilon}, W_1^{\epsilon'}}(\pi)$ has $L$-parameter
 \[  \theta(\phi) = (\phi \otimes \chi^{-1}) \oplus  \chi. \]

 \item For each $\epsilon = \pm 1$, the theta correspondence $\pi \mapsto \theta_{\psi, V_{2}^{\epsilon}, W_1^{\epsilon'}}(\pi)$ gives a bijection
\[  \Pi_{\phi} \longleftrightarrow \Pi_{\theta(\phi)}^{\epsilon}.  \]  

\end{enumerate} 

\item Suppose that $\phi$ contains $\chi^2$.

\begin{enumerate}

\item For any fixed $\pi \in \Pi_{\phi}^{\epsilon'}$, exactly one of $\Theta_{\psi, V_{2}^+, W_1^{\epsilon'}}(\pi)$ or $\Theta_{\psi, V_{2}^-, W_1^{\epsilon'}}(\pi)$ is nonzero.

\item If $\Theta_{\psi, V_{2}^{\epsilon}, W_1^{\epsilon'}}(\pi)$ is nonzero, then $\theta_{\psi, V_{2}^{\epsilon}, W_1^{\epsilon'}}(\pi)$ has $L$-parameter
 \[  \theta(\phi) = (\phi \otimes \chi^{-1} ) \oplus  \chi. \]

\item The theta correspondence $\pi \mapsto \theta_{\psi, V_{2}^{\epsilon}, W_1^{\epsilon'}}(\pi)$ gives a bijection
 \[  \Pi_{\phi} \longleftrightarrow \Pi_{\theta(\phi)}. \]

\end{enumerate} 
\item
\begin{itemize}

\item If $\phi \ne \chi^2$, let 
\[  S_{\phi}=(\ZZ/2\ZZ)b_1 \ , \ S_{\theta(\phi)}  = (\ZZ/2\ZZ)b_1 \times (\ZZ /2 \ZZ) b_2, \]
where the extra copy of $\ZZ/2\ZZ$ of $S_{\theta(\phi)}$ arises from the summand $\chi$ in $\theta(\phi)$. \\
Using the two bijections \[  J^{\psi}(\phi):  \Pi_{\phi} \longleftrightarrow \Irr(S_{\phi})   \quad \text{and} \quad J^{\psi}(\theta(\phi)):  \Pi_{\theta(\phi)} \longleftrightarrow \Irr(S_{\theta(\phi)}), \]
we obtain a bijection
\begin{align*}
 \Irr(S_{\phi}) & \longleftrightarrow \Irr^{\epsilon}(S_{\theta(\phi)}) \\
 \eta & \longleftrightarrow \theta(\eta)
\end{align*}
induced by the theta correspondence, where $\Irr^{\epsilon}(S_{\theta(\phi)})$ is the set of irreducible characters $\eta'$ of $S_{\theta(\phi)}$ such that $\eta'(z_{\theta(\phi)}) =\eta'((b_1,1)) \cdot \eta'((1,b_2))= \epsilon$, and the bijection is determined by
 \[  \theta(\eta)|_{S_{\phi}}  =\eta. \]

\item If $\phi = \chi^2$, then $\phi \otimes \chi^{-1}= \chi$, and so
 \[
   S_{\theta(\phi)}  =  S_{\phi}. \]
Thus, one has a canonical bijection
\begin{align*}
  \Irr(S_{\phi}) & \longleftrightarrow \Irr(S_{\theta(\phi)}) \\
  \eta & \longleftrightarrow \theta(\eta)
\end{align*}
induced by the theta correspondence and it is given by \[ \theta(\eta) = \eta. \]

\end{itemize}

\item If $\Theta_{\psi, V_{2}^{\epsilon}, W_1^{\epsilon'}}(\pi)$ is nonzero, then $\Theta_{\psi, V_{2}^{\epsilon}, W_1^{\epsilon'}}(\pi)$ is irreducible and so  $\Theta_{\psi, V_{2}^{\epsilon}, W_1^{\epsilon'}}(\pi)=\theta_{\psi, V_{2}^{\epsilon}, W_1^{\epsilon'}}(\pi).$

\end{enumerate}

\end{thm}

\subsection{Case (iii)} Now we shall consider the theta correspondence for $\U(V_{3}^\epsilon) \times \U(W_1^{\epsilon'})$. The following summarizes some results of \cite{Ge1}, \cite{Ge2}, \cite{Ha}.
\begin{thm}\label{theta}Let $\phi$ be a $L$-parameter of $U(W_1^{\pm})$. Then we have:

\begin{enumerate}
\item For any $\epsilon,\epsilon'=\pm1$ and any $\pi \in \Pi_{\phi}^{\epsilon'}$, $\Theta_{\psi, V_3^{\epsilon},W_1^{\epsilon'}}(\pi)$ is nonzero and irreducible.

\item $\Theta_{\psi, V_3^{\epsilon},W_1^{\epsilon'}}(\pi)$ $=\begin{cases} \text{a non-tempered representation} & \text{if $\epsilon(\frac{1}{2},\phi \otimes \chi^{-3}, \psi_2^{E})=\epsilon \cdot \epsilon'$} \\ \text{a supercuspidal representation } & \text{ if  $\epsilon(\frac{1}{2},\phi \otimes \chi^{-3}, \psi_2^{E})=-\epsilon \cdot \epsilon'$}. \end{cases} $

\item The $L$-parameter $\theta(\phi)$ of $\Theta_{\psi, V_3^{\epsilon},W_1^{\epsilon'}}(\pi)$ has the following two forms;
\\ $$\theta(\phi)=\begin{cases} \theta_1(\phi)=\chi  |\cdot|_{E}^{\frac{1}{2}} \oplus \phi \cdot \chi^{-2} \oplus \chi  |\cdot|_{E}^{-\frac{1}{2}}& \text{if $\epsilon(\frac{1}{2},\phi \otimes \chi^{-3}, \psi_2^{E})=\epsilon \cdot \epsilon'$} \\ \theta_2(\phi)= \phi \cdot \chi^{-2} \oplus \chi \boxtimes \textbf{S}_2 &  \text{if $\epsilon(\frac{1}{2},\phi \otimes \chi^{-3}, \psi_2^{E})=-\epsilon \cdot \epsilon'$}  \end{cases},$$ where $\textbf{S}_2:SL_2(\CC) \to GL_2(\CC)$ is the tautological 2-dimensional representation of $SL_2(\mathbb{C}).$

\end{enumerate}
\begin{proof}The property (i) follows from the Proposition 2.5.1 (b) in \cite{Ge2}. The Proposition 5.2.2 in \cite{Ge2} asserts that $\Theta_{\psi,V_3^{\epsilon},W_1^{\epsilon'}}(\pi)$ is supercuspidal if and only if $\Theta_{\psi,V_1^{\epsilon},W_1^{\epsilon'}}(\pi\otimes (\chi_{1})^{-1})=0$ where $\chi_1$ is the restriction of $\chi$ to $E^1$. Since the base change of $\chi_1$ to $GL_1(E)$ is defined by $\chi_E(x):=\chi_1(\frac{x}{\bar{x}})=\chi^2(x)$, the $L$-parameter of $\pi\otimes (\chi_1)^{-1}$ is $\phi \otimes \chi^{-2}$. Thus by the Theorem \ref{Te} (ii), $$\Theta_{\psi,V_1^{\epsilon},W_1^{\epsilon'}}(\pi\otimes (\chi_{1})^{-1})=0  \Longleftrightarrow \epsilon(\frac{1}{2},\phi \otimes \chi^{-3}, \psi_2^{E})=-\epsilon \cdot \epsilon'$$  and this accounts for the property (ii). The property (iii) follows from the description of $L$-parameters of $\Theta_{\psi, V_3^{\epsilon},W_1^{\epsilon'}}(\pi)$ in \cite[\S7]{Ha}. (see page 985) 
\end{proof}

\end{thm}
The following theorem explicates a precise local theta correspondence between $(U(W_1^{\epsilon'}),U(V_3^{\epsilon}))$. The proof of this will be given in Section 5.

\begin{thm}\label{thm1}Let $\phi$ be a $L$-parameter of $U(W_1^{\pm})$ and for $\pi \in \Pi_{\phi}^{\epsilon'}$, let $\theta_1(\phi), \ \theta_2(\phi)$ be the two possible $L$-parameters of $\Theta_{\psi, V_3^{\epsilon},W_1^{\epsilon'}}(\pi)$ as above. Then we have
\begin{enumerate}
\item For $\epsilon, \epsilon'$ such that $\epsilon(\frac{1}{2},\phi \otimes \chi^{-3}, \psi_2^{E})=\epsilon \cdot \epsilon'$, the theta correspondence $\pi \mapsto \theta_{\psi, V_{3}^{\epsilon}, W_1^{\epsilon'}}(\pi)$ gives a bijection
\[  \Pi_{\phi} \longleftrightarrow \Pi_{\theta_1(\phi)}.  \]  

\item For $\epsilon, \epsilon'$ such that $\epsilon(\frac{1}{2},\phi \otimes \chi^{-3}, \psi_2^{E})=-\epsilon \cdot \epsilon'$, the theta correspondence $\pi \mapsto \theta_{\psi, V_{3}^{\epsilon}, W_1^{\epsilon'}}(\pi)$ gives an injection
\[  \Pi_{\phi} \hookrightarrow \Pi_{\theta_2(\phi)}.  \]  

\end{enumerate}
Write \begin{itemize} \item $S_{\phi}\ = \ (\ZZ/2\ZZ)a_1$
\item $S_{\theta_1(\phi)}=(\ZZ/2\ZZ)a_1 \quad \ \quad\quad\quad\quad\quad \text{ if  }\ \epsilon(\frac{1}{2},\phi \otimes \chi^{-3}, \psi_2^{E})=\epsilon \cdot \epsilon'$
 \item $S_{\theta_2(\phi)}=(\ZZ/2\ZZ)a_1   \times (\ZZ/2\ZZ)a_2 
 \quad \text{ if  }\ \epsilon(\frac{1}{2},\phi \otimes \chi^{-3}, \psi_2^{E})=-\epsilon \cdot \epsilon' .$\end{itemize}

\noindent (Note $\theta_2(\phi)$ is the square-integrable $L$-parameter of $U(V_3^{\epsilon})$ and so $(\ZZ/2\ZZ)a_2$ of $S_{\theta_2(\phi)}$ arises from the summand $\chi \boxtimes \textbf{S}_2$ in $\theta_2(\phi)$.)

Using three bijections

\begin{itemize}
\item $J^{\psi}(\phi): \Pi_{\phi} \longleftrightarrow \Irr(S_{\phi}) $

\item $J^{\psi}(\theta_1(\phi)) : \Pi_{\theta_1(\phi)} \longleftrightarrow \Irr(S_{\theta_1(\phi)}) $

\item $J^{\psi}(\theta_2(\phi)) : \Pi_{\theta_2(\phi)} \longleftrightarrow \Irr(S_{\theta_1(\phi)}),$

\end{itemize}

the following bijection and inclusion induced by the theta correspondence 
\begin{align*}
 \Irr(S_{\phi}) & \longleftrightarrow \Irr(S_{\theta_1(\phi)}) \\
 \eta & \longleftrightarrow \theta_1(\eta),
\end{align*}  
\begin{align*}
 \Irr(S_{\phi}) & \hookrightarrow \Irr(S_{\theta_2(\phi)}) \\
 \eta & \mapsto \theta_2(\eta)
\end{align*}can be explicated as follows:
\begin{equation}\label{theta1} \theta_{1}(\eta)(a_1)=\eta(a_1) \cdot \epsilon(\frac{1}{2},\phi \otimes \chi^{-3},\psi_{2}^E ),\end{equation}
 \begin{equation}\label{theta2}\theta_{2}(\eta)(a_1)=\eta(a_1) \cdot \epsilon(\frac{1}{2},\phi \otimes \chi^{-3},\psi_{2}^E ) , \\ \quad \quad \theta_{2}(\eta)(a_2)=-1.\end{equation}
\end{thm}

\section{\textbf{Main Theorem}}

In this section, we prove our main theorem. Prior to stating our main theorem, we first elaborate on the results of Plessis and Gan-Ichino for the Gan--Gross--Prasad conjecture for unitary groups, some of which we shall use in the proof of our main theorem. In \cite[\S3]{iw}, Gan and Ichino have made a neat exposition on this, we quote their treatment here. Throughout this section, we fix a nontrivial additive character $\psi:F \to \CC^{\times}$ and make a tacit use $J^{\psi}$ for the bijection in the local Langlands correspondence.
\subsection{Pairs of spaces}

To explain both the \textbf{(Bessel)} and \textbf{(Fourier-Jacobi)} cases of the Gross--Prasad conjecture simulataneously, we consider the pair of spaces:
\[  V_n^+ \subset V_{n+1}^+ \quad \text{or} \quad W_n^+ = W_n^+. \]
Then their relevant pure inner form (other than itself) are
\[  V_n^- \subset V_{n+1}^-  \quad \text{or} \quad W_n^- = W_n^- .\]

\noindent For $a \in F^{\times}$, if we set $L_a$ denotes the a 1-dimensional Hermitian space with form $a \cdot \N_{E/F}$, then
\[  V_{n+1}^{\epsilon}/V_n^{\epsilon}   \cong L_{(-1)^n}. \]
Write \[  G_n^\epsilon =  \U(V_{n+1}^\epsilon) \times \U(V_{n}^\epsilon) \quad \text{or} \quad \U(W_n^\epsilon) \times \U(W_n^\epsilon), \]
\[   H_n^\epsilon = \U(V_n^\epsilon) \quad \text{or} \quad \U(W_n^\epsilon), \]
and we have a diagonal embedding
\[ \Delta : H_n^\epsilon \hookrightarrow G_n^\epsilon.\]

For an $L$-parameter $\phi = \phi^{\diamondsuit} \times \phi^{\heartsuit}$ for $G_n^\pm$, its associated component group is:
\[  S_{\phi}  = S_{\phi^{\diamondsuit}}  \times S_{\phi^{\heartsuit}}. \]

Let $\eta$ be a character of $S_{\phi}$. Then under the local Langlands correspondence, the representation $\pi(\eta) \in \Pi_{\phi}$ is a representation of a relevant pure inner form if and only if
\[  \eta(z_{\phi^{\diamondsuit}},1) =\eta(1,z_{\phi^{\heartsuit}}) , \]
and $\pi(\eta)$ is a representation of $G_n^{\epsilon}$ if and only if 
\[  \eta(z_{\phi^{\diamondsuit}}, 1)  = \eta(1,z_{\phi^{\heartsuit}})  = \epsilon. \]

\subsection{The recipe} 
\label{SS:eta}In this subsection, we define the distinguished characters for the recipe of the GGP conjecture. For an $L$-parameter $\phi =  \phi^{\diamondsuit} \times \phi^{\heartsuit}$ of $G_n$, write \[  S_{\phi^{\diamondsuit}}  = \prod_i  (\ZZ /2 \ZZ) a_i \quad \text{and} \quad S_{\phi^{\heartsuit}} = \prod_j  (\ZZ / 2\ZZ) b_j. \]
Then $\eta \in S_{\phi}$ is completely determined by the signs $\eta(a_i) = \pm 1$ and $\eta(b_j) = \pm 1$.

We define the relevant distinguished characters of $S_{\phi}$ for the Bessel and Fourier--Jacobi cases as follows:

\begin{enumerate}

\item \textbf{Bessel case.}
We set $\psi^E_{-2}(x)  = \psi(-\Tr_{E/F}(\delta x))$ and define $\eta^{\spadesuit}\in \text{Irr}(S_{\phi})$ as follows:
\[
\begin{cases}
 \eta^{\spadesuit}(a_i)  = \epsilon(\frac{1}{2}, \phi^{\diamondsuit}_i \otimes \phi^{\heartsuit}, \psi^E_{-2}); \\
 \eta^{\spadesuit}(b_j)  = \epsilon(\frac{1}{2},  \phi^{\diamondsuit} \otimes \phi^{\heartsuit}_{j}, \psi^E_{-2}).
\end{cases}
\]
\\
\item \textbf{Fourier--Jacobi case.} We set $\psi^E_2(x)  = \psi(\Tr_{E/F}(\delta x))$ and $\psi^E(x)  = \psi(\frac{1}{2}\Tr_{E/F}(\delta x))$. The distinguished character $\eta^{\clubsuit}$ of $S_{\phi}$ depends on the parity of $n = \dim W_n^\epsilon$ as follows:\\

\begin{itemize}
\item If $n$ is odd, we set
\[ \begin{cases}
 \eta^{\clubsuit}(a_i) = \epsilon(\frac{1}{2}, \phi^{\diamondsuit}_{i} \otimes \phi^{\heartsuit} \otimes \chi^{-1},  \psi_2^E);  \\
  \eta^{\clubsuit}(b_j)  = \epsilon(\frac{1}{2}, \phi^{\diamondsuit} \otimes \phi^{\heartsuit}_{j} \otimes \chi^{-1}, \psi_2^E). \end{cases} \]

\item If $n$ is even, we set
\[  \begin{cases}
\eta^{\clubsuit}(a_i) = \epsilon(\frac{1}{2}, \phi^{\diamondsuit}_i \otimes \phi^{\heartsuit} \otimes \chi^{-1},  \psi^E); \\
  \eta^{\clubsuit}(b_j)  = \epsilon(\frac{1}{2}, \phi^{\diamondsuit} \otimes \phi^{\heartsuit}_{j} \otimes \chi^{-1}, \psi^E). \end{cases} \]
\end{itemize}
\end{enumerate}

\subsection{Theorem (B) and (FJ) for generic parameter}\label{gen}

We state the results of Plessis(\cite{bp1}, \cite{bp2}, \cite{bp3}) and Gan-Ichino(\cite{iw}) on the GGP conjecture.

\begin{itemize}

\item[(B)$_n$]
For a \emph{generic} $L$-parameter $\phi$ for $G_n^\pm = \U(V_{n+1}^\pm) \times \U(V_{n}^\pm)$ and a representation $\pi(\eta) \in \Pi_{\phi}$ of a relevant pure inner form $G_n^\epsilon$,
\[  \Hom_{\Delta H_n^\epsilon}(\pi(\eta), \CC) \ne 0 \Longleftrightarrow  \eta = \eta^{\spadesuit}. \]

\item[(FJ)$_n$]
For a \emph{generic} $L$-parameter $\phi$ for $G_n^\pm = \U(W_n^\pm) \times \U(W_n^\pm)$ and a representation $\pi(\eta) \in \Pi_{\phi}$ of a relevant pure inner form $G_n^\epsilon$, 
\[ \Hom_{\Delta H_n^\epsilon}(\pi(\eta), \nu) \ne 0 \Longleftrightarrow \eta = \eta^{\clubsuit}. \]

\end{itemize}

We shall denote by (B) the collection of statements (B)$_n$ for all $n\ge0$, and by (FJ) the collection of statements (FJ)$_n$ for all $n\ge0$.

The main theorem of this paper investigates (B)$_2$ of the conjecture for some endoscopic $L$-parameters of $U(V_3^{\pm})\times U(V_2^{\pm})$ involving a non-generic (and thus non-tempered) parameter of $U(V_3^{\pm})$. Now we state our main theorem.

\begin{thm}\label{thm2} Let $\phi_1,\phi_2$ be a $L$-parameter of $U(W_1^{\pm})$ such that $\phi_2 \ne \chi^2$ and let $$\theta_1(\phi_1)=\chi  |\cdot|_{E}^{\frac{1}{2}} \oplus \phi_1 \otimes \chi^{-2} \oplus \chi  |\cdot|_{E}^{-\frac{1}{2}},$$ $$\theta_2(\phi_1)=\phi_1 \cdot \chi^{-2} \oplus \chi \boxtimes \textbf{S}_2$$ be the two $L$-parameters of $U(V_3^{\pm})$ appeared in Theorem \ref{theta} (iii) and let $$\theta(\phi_2)=\phi_2 \otimes \chi^{-1} \oplus \chi $$ be the $L$-parameters of $U(V_2^{\pm})$ appeared in Theorem \ref{Tae} (ii).\\
Write $$\begin{cases} S_{\theta_1(\phi_1)}=(\ZZ/2\ZZ)a_1; \\ S_{\theta_2(\phi_1)}=(\ZZ/2\ZZ)a_1 \times (\ZZ/2\ZZ)a_2;  \\ S_{\theta(\phi_2)}=(\ZZ/2\ZZ)b_1 \times (\ZZ/2\ZZ)b_2 . \end{cases}$$\\
\noindent Then for $i=1,2$, and for $(\pi_3^{i},\pi_2^{i})\in \Pi_{\theta_i(\phi_1)} \times \Pi_{\theta(\phi_2)},$ 
$$\Hom_{U(V_2^{\epsilon})}(\pi_3^{i},\pi_2^{i})\ne0 \Leftrightarrow (\pi_3^{i},\pi_2^{i})=(\pi_{\theta_i(\phi_1)}(\eta_{i}^{\diamondsuit}),\pi_{\theta(\phi_2)}(\eta_i^{\heartsuit}))$$ where $(\eta_i^{\diamondsuit},\eta_{i}^{\heartsuit})\in \Irr(S_{\theta_i(\phi_1)}) \times  \Irr(S_{\theta(\phi_2)})$ the pair of characters of the component group is specified as follows;\\
\begin{equation}\label{eta1}\begin{cases}\eta_1^{\diamondsuit}(a_1)=\epsilon(\frac{1}{2},\phi_2^{-1} \otimes  \phi_1 \otimes \chi^{-1} ,\psi_2^E)\cdot \epsilon(\frac{1}{2},\phi_1 \otimes \chi^{-3},\psi_{2}^E ); \\\eta_1^{\heartsuit}(b_1)=\epsilon(\frac{1}{2},\phi_2^{-1} \otimes  \phi_1 \otimes \chi^{-1} ,\psi_2^E);\\ \eta_1^{\heartsuit}(b_2)=\epsilon(\frac{1}{2},\phi_1 \otimes \chi^{-3},\psi_{2}^E ),  \end{cases}\end{equation} and 
\begin{equation}\label{eta2}\begin{cases}\eta_2^{\diamondsuit}(a_1)=\epsilon(\frac{1}{2},\phi_2^{-1} \otimes  \phi_1 \otimes \chi^{-1} ,\psi_2^E)\cdot \epsilon(\frac{1}{2},\phi_1 \otimes \chi^{-3},\psi_{2}^E );\\ \eta_2^{\diamondsuit}(a_2)=-1; \\ \eta_2^{\heartsuit}(b_1)=\epsilon(\frac{1}{2},\phi_2^{-1} \otimes  \phi_1 \otimes \chi^{-1} ,\psi_2^E); \\ \eta_2^{\heartsuit}(b_2)=-\epsilon(\frac{1}{2},\phi_1 \otimes \chi^{-3},\psi_{2}^E ).\end{cases}\end{equation}

\end{thm}

\begin{proof}In this proof, we assume Theorem \ref{thm1} whose proof will be given in the next section.

We first prove the existence of some $\epsilon$ and $(\pi_3^{i},\pi_2^{i})\in \Pi^{\epsilon}_{\theta_i(\phi_1)} \times \Pi^{\epsilon}_{\theta(\phi_2)}$ such that $\Hom_{U(V_2^{\epsilon})}(\pi_3^{i},\pi_2^{i})\ne0$ for each $i=1,2$.  

One has the see-saw diagram : ($\epsilon, \epsilon'$ will be determined soon)
\[
 \xymatrix{
  \U(W_1^{\epsilon'})  \times \U(W_1^{\epsilon'})  \ar@{-}[dr] \ar@{-}[d] & \U(V^{\epsilon}_{3}) \ar@{-}[d] \\
  \U(W_1^{\epsilon'}) \ar@{-}[ur] &  \U(V^{\epsilon}_2) \times \U(L_{1})}.
\]
We consider the three theta correspondence in this diagram : 
\begin{enumerate}
\item $U(V_3^{\epsilon}) \times U(W_1^{\epsilon})$ relative to the pair of characters $(\chi, \chi^3)$;
\item $\U(V^{\epsilon}_2) \times U(W^{\epsilon'}_1)$ relative to the pair of characters $(\chi, \chi^2)$;
\item $\U(L_{1}) \times \U(W_1^{\epsilon'})$ relative to the pair of characters $(\chi, \chi)$.

\end{enumerate}

\noindent Let us take $\epsilon'=\epsilon(\frac{1}{2},\phi_2^{-1} \otimes  \phi_1 \otimes \chi^{-1} ,\psi_2^E).$ Then by (FJ)$_1$, we have $$\Hom_{U(W_1^{\epsilon'})}(\pi_{\phi_2^{-1}}\otimes \pi_{\phi_1},\omega_{\psi,W_1^{\epsilon'}})\ne0.$$
(here, $\pi_{\phi_2^{-1}}, \pi_{\phi_1} \in \Pi^{\epsilon'}_{\phi^{-1}_2} \times \Pi^{\epsilon'}_{\phi_1}$ and note  that the set  $\Pi^{\epsilon'}_{\phi^{-1}_2} \times \Pi^{\epsilon'}_{\phi_1}$ is singleton.)\\
Since $\pi_{\phi_2^{-1}}, \pi_{\phi_1}$ are both unitary, one has $$\Hom_{U(W_1^{\epsilon'})}\big( (\pi_{\phi_2^{-1}})^{\vee}\otimes \omega_{\psi,W_1^{\epsilon'}},\pi_{\phi_1})\ne0$$ and the $L$-parameter of $(\pi_{\phi_2^{-1}})^{\vee}$ is $\phi_2$.\\
For any $\epsilon=\pm1$, Theorem \ref{Tae} (i) asserts that there is $\tau \in \Pi^{\epsilon}_{\theta(\phi_2)}$ such that $\Theta_{\psi,W^{\epsilon'}_1,V^{\epsilon}_2}(\tau)=(\pi_{\phi^{-1}_2})^{\vee}.$
Then by the see-saw identity, one has $$0 \ne \Hom_{U(W_1^{\epsilon'})}\big( (\pi_{\phi_2^{-1}})^{\vee}\otimes \omega_{\psi,W_1^{\epsilon'}},\pi_{\phi_1})=\Hom_{U(V_2^{\epsilon})}(\Theta_{V^{\epsilon}_3,W^{\epsilon'}_1}(\pi_{\phi_1}),\tau).$$
\noindent By Theorem \ref{thm1}, the $L$-parameter of $\Theta_{V^{\epsilon}_3,W^{\epsilon'}_1}(\pi_{\phi_1})$ depends on the $\epsilon$ as follows :
\begin{equation}\label{par}\text{ the } L \text{-parameter of } \Theta_{V^{\epsilon}_3,W^{\epsilon'}_1}(\pi_{\phi_1})= \begin{cases} \theta_1(\phi_1) & \text{ if  } \epsilon = \epsilon' \cdot \epsilon(\frac{1}{2},\phi_1 \otimes \chi^{-3},\psi_{2}^E ), \\ \theta_2(\phi_1) & \text{ if } \epsilon = -\epsilon' \cdot \epsilon(\frac{1}{2},\phi_1 \otimes \chi^{-3},\psi_{2}^E ). \end{cases}\end{equation}
Thus we proved the existence $(\pi_3^{i},\pi_2^{i})\in \Pi^{\epsilon}_{\theta_{i}(\phi_1)} \times \Pi^{\epsilon}_{\theta(\phi_2)}$ such that $\Hom_{U(V_2^{\epsilon})}(\pi_3^{i},\pi_2^{i})\ne0$ for each $i=1,2$. \\ 

\noindent Next, we shall show that such a pair $(\pi_3^{i},\pi_2^{i})$ is unique and their characters of the component group are exactly the same one suggested in (\ref{eta1}), (\ref{eta2}).

For $i=2$, the uniqueness directly follows from (B)$_2$ because $\theta^{(2)}(\phi_1)$ is a tempered $L$-parameter.
(Note that supercuspidal $L$-parameter of $U(V_3^{\epsilon})$ is tempered because the center of $U(V_3^{\epsilon})$ is compact.)

Thus we if set $$\epsilon'=\epsilon(\frac{1}{2},\phi_2^{-1} \otimes  \phi_1 \otimes \chi^{-1} ,\psi_2^E),$$ $$\epsilon=-\epsilon(\frac{1}{2},\phi_2^{-1} \otimes  \phi_1 \otimes \chi^{-1} ,\psi_2^E)\cdot \epsilon(\frac{1}{2},\phi_1 \otimes \chi^{-3},\psi_{2}^E ),$$ the pair $(\Theta_{V^{\epsilon}_3,W^{\epsilon'}_1}(\pi_{\phi_1}),\tau)$ in the above argument is the very one that makes $\Hom_{U(V_2^{\epsilon})}(\pi_3^{2},\pi_2^{2})\ne0$.\\ If we combine this with Theorem \ref{Tae} and Theorem \ref{thm1}, we can easily check that their associated characters of the component group are $\eta_2^{\diamondsuit}, \eta_2^{\heartsuit}$ in (\ref{eta2}).\\

Now, we suppose $\Hom_{U(V_2^{\epsilon})}(\pi_3^{1},\pi_2^{1})\ne0$ for some $(\pi_3^{1},\pi_2^{1})\in \Pi_{\theta^{(1)}(\phi_1)} \times \Pi_{\theta(\phi_2)}.$ \\
By Theorem \ref{thm1}, we can write $\pi_3^{1}=\Theta_{\psi, V_3^{\epsilon},W_1^{\epsilon'}}(\sigma)$ for some $\sigma \in \Pi_{\phi_1}^{\epsilon'}$ and using the see-saw identity, one has $$\Hom_{U(V_2^{\epsilon})}(\pi_3^{1},\pi_2^{1})=\Hom_{U(W_1^{\epsilon'})}(\Theta_{\psi, W_1^{\epsilon'},V_2^{\epsilon}}(\pi_2^{1}) \otimes \omega_{\psi,W_1^{\epsilon'}},\sigma )\ne 0.$$ In particular, we see $\Theta_{\psi, W_1^{\epsilon'},V_2^{\epsilon}}(\pi_2^{1}) \ne 0$.\\
Since $\Theta_{\psi, W_1^{\epsilon'},V_2^{\epsilon}}(\pi_2^{1})$ and $\sigma$ are both unitary, we have $$\Hom_{U(W_1^{\epsilon'})}(\Theta_{\psi, W_1^{\epsilon'},V_2^{\epsilon}}^{\vee}(\pi_2^{1}) \otimes \sigma,\omega_{\psi,W_1^{\epsilon'}} )\ne0$$ and by Theorem \ref{Tae} (i), the $L$-parameter of $\Theta_{\psi, W_1^{\epsilon'},V_2^{\epsilon}}^{\vee}(\pi_2^{1})$ is $\phi_2^{-1}$. Thus by (FJ)$_1$, we see that $$\epsilon'=\epsilon(\frac{1}{2},\phi_2^{-1} \otimes  \phi_1 \otimes \chi^{-1} ,\psi_2^E).$$
By Theorem \ref{Tae} and Theorem \ref{thm1}, one has $$\eta_1^{\diamondsuit}(a_1)=\epsilon=\epsilon(\frac{1}{2},\phi_2^{-1} \otimes  \phi_1 \otimes \chi^{-1} ,\psi_2^E)\cdot \epsilon(\frac{1}{2},\phi_1 \otimes \chi^{-3},\psi_{2}^E ),$$ $$\eta_1^{\heartsuit}(b_1)=\epsilon(\frac{1}{2},\phi_2^{-1} \otimes  \phi_1 \otimes \chi^{-1} ,\psi_2^E),$$  $$\eta_1^{\heartsuit}(b_2)=\epsilon(\frac{1}{2},\phi_1 \otimes \chi^{-3},\psi_{2}^E ).$$

\noindent (The third equality follows from $\eta_1^{\heartsuit}(z_{\theta(\phi_2)})=\eta_1^{\heartsuit}(b_1)\cdot \eta_1^{\heartsuit}(b_2)=\epsilon.$)

\end{proof}

\begin{rem} Recall that an $L$-parameter $\phi$ of $U(V_n^{\pm})$ or $U(W_n^{\pm})$ is called \emph{generic} if its adjoint $L$-function $$L(s,Ad\circ \phi)=L(s,As^{(-1)^n}\circ \phi)$$ is holomorphic at $s=1$. (Here, Ad is the adjoint representation of $^LU(n)$ on its Lie algebra Lie($^LU(n)$).) A conjecture of Gross, Prasad and Rallis \cite[\S 6]{GP} predicts that $\phi$ is generic if and only if  its associated $L$-packet $\Pi_{\phi}$ has a generic representation, and quite recently, Gan and Ichino (Proposition B.1 in the Appendix in \cite{iw}) proved this under the hypothesis of some properties of local Langlands correspondence. Since the Corollary 4.2.3 in \cite{Ge2} asserts that all elements in $\Pi_{\theta_1(\phi_1)}$ have no Whittaker models, we see that $\theta_1(\phi_1)$ is non-generic. 
\end{rem}
\begin{rem} As we mentioned in Remark \ref{rem1.1}, the GGP-conjecture is no longer true for non-generic $L$-parameter and Theorem \ref{thm2} shows that for having $\Hom_{U(V_2^{\epsilon})}(\pi_3,\pi_2)\ne0$, the $L$-parameter of $\pi_2$ should be of very special form, namely, a theta lift from $U(W_1^{\pm})$. 

\end{rem}

\begin{rem}\label{rem1}If $\pi_3=\Theta_{\psi,V_3^{\epsilon},W_1^{\epsilon'}}(\phi_1)$, $\pi_2=\Theta_{\psi,V_2^{\epsilon},W_1^{\epsilon'}}(\phi_2)$ for some $L$-parameters $\phi_1,\phi_2$ of $U(W_1^{\epsilon'})$, then the above result is condensed into one sentence as follows:
$$\Hom_{U(V_2^{\epsilon})}(\pi_3,\pi_2)\ne 0 \text{ if and only if } \epsilon'=\epsilon(\frac{1}{2},\phi_2^{-1} \otimes  \phi_1 \otimes \chi^{-1} ,\psi_2^E).$$

\end{rem}

\begin{rem}In the proof of Theorem \ref{thm2}, the unique pair of representations $(\pi_3^{2},\pi_2^{2})\in \Pi_{\theta_2(\phi_1)} \times \Pi_{\theta(\phi_2)}$ such that $$\Hom_{U(V_2^{\epsilon})}(\pi_3^{2},\pi_2^{2})\ne0 $$ are obtained by the theta lifts from $U(W_1^{\epsilon'})$. However, if $\phi_2 = \chi^2$, the following proposition says that this is not true.
\end{rem}

\begin{prop}We retain the same notation as in Theorem \ref{thm2} except for $\phi_2=\chi^2$ so that $$\theta(\phi_2)=\chi \oplus \chi$$ and $$S_{\theta(\phi_2)}=(\ZZ/2\ZZ)b_1.$$\\ Then, $$\Hom_{U(V_2^{\epsilon})}(\pi_3^{1},\pi_2^{1})\ne0 \Leftrightarrow (\pi_3^{1},\pi_2^{1})=\big(\pi_{\theta_1(\phi_1)}(\eta_{1}^{\diamondsuit}),\pi_{\theta(\phi_2)}(\eta_1^{\heartsuit})\big)$$ where \begin{equation}\label{phi}\begin{cases}\eta_1^{\diamondsuit}(a_1)=1, \ \\ \eta_1^{\heartsuit}(b_1)=\epsilon(\frac{1}{2},\phi_1 \otimes \chi^{-3},\psi_{2}^E ).\end{cases}\end{equation}

\noindent Furthermore, $(\pi_3^{2}(\eta^{ \diamondsuit}_2),\pi_2^{2}(\eta^{\heartsuit}_2))\in \Pi_{\theta^{(2)}(\phi_1)} \times \Pi_{\theta(\phi_2)}$, suggested in the recipe of (B)$_2$ in this case, does not come from the theta lifts of $U(W^{\pm}_1)$.
\end{prop}

\begin{proof}We first note that $z_{\theta(\phi_2)}$, the image of $-1\in GL_2(\mathbb{C})$ in $S_{\theta(\phi_2)}$, is $2b_1$ and so for every $\eta \in S_{\theta(\phi_2)}$, one has $\eta(z_{\theta(\phi_2)})=1$ and $\Pi_{\theta(\phi_2)}=\Pi^{1}_{\theta(\phi_2)}.$ Thus every nonzero $\pi_2^1 \in \Pi_{\theta(\phi_2)}$ is indeed in $ \Pi^{1}_{\theta(\phi_2)}.$\\ By Theorem \ref{thm1}, all element $\pi_3^1 \in \Pi_{\theta^{(1)}(\phi_1)} $ can be written $\pi_3^1=\Theta_{\psi, V_3^{\epsilon},W_1^{\epsilon'}}(\sigma)$ for some $\epsilon, \epsilon'\in \{\pm1\}$ such that $\epsilon \cdot \epsilon'=\epsilon(\frac{1}{2},\phi_1 \otimes \chi^{-3},\psi_{2}^E )$ and $\sigma \in \Pi^{\epsilon'}_{\phi_1}$.
\\Thus in order to have $\Hom_{U(V_2^{\epsilon})}(\pi_3^{1},\pi_2^{1})\ne 0$, there is no choice but to choose $\epsilon=1$ and so $\epsilon'$ should be $\epsilon(\frac{1}{2},\phi_1 \otimes \chi^{-3},\psi_{2}^E )$ when writing $\pi_3^{1}$ as the theta lift from $U(W_1^{\epsilon'})$.\\ With these choices of $\epsilon=1, \epsilon'=\epsilon(\frac{1}{2},\phi_1 \otimes \chi^{-3},\psi_{2}^E )$, the see-saw identity gives $$\Hom_{U(V_2^{\epsilon})}(\pi_3^{1},\pi_2^{1})\simeq \Hom_{U(W_1^{\epsilon'})}(\Theta_{\psi, W_1^{\epsilon'},V_2^{\epsilon}}(\pi_2^{1}) \otimes \omega_{\psi,W_1^{\epsilon'}},\sigma ).$$ For having $\Hom_{U(V_2^{\epsilon})}(\pi_3^{1},\pi_2^{1})\ne 0$, $\Theta_{\psi, W_1^{\epsilon'},V_2^{\epsilon}}(\pi_2^{1}) $ should be nonzero and in view of Theorem \ref{Tae} (ii), $\pi_2^1$ should be $\Theta_{V_2^{\epsilon}, W_{1}^{\epsilon'}}(\tau)$, where $\tau$ is the unique nonzero representation $\tau$ in $\Pi^{\epsilon'}_{\phi_2}$. (Note that $\Pi^{\epsilon'}_{\phi_2}$ is singleton.)
\\Since $\Hom^{\vee}_{U(W_1^{\epsilon'})}(\Theta_{\psi, W_1^{\epsilon'},V_2^{\epsilon}}(\pi_2^{1}) \otimes \omega_{\psi,W_1^{\epsilon'}},\sigma) \simeq \Hom_{U(W_1^{\epsilon'})}(\Theta^{\vee}_{\psi, W_1^{\epsilon'},V_2^{\epsilon}}(\pi_2^{1}) \otimes \sigma ,\omega_{\psi,W_1^{\epsilon'}} )$ and $$\Theta^{\vee}_{\psi, W_1^{\epsilon'},V_2^{\epsilon}}(\pi_2^{1})=\tau^{\vee}$$ has $L$-parameter $\phi_2^{-1}$, our choice $\epsilon'=\epsilon(\frac{1}{2},\phi_1 \otimes \chi^{-3},\psi_{2}^E )$ and (FJ)$_1$ implies $$\Hom_{U(W_1^{\epsilon'})}(\Theta^{\vee}_{\psi, W_1^{\epsilon'},V_2^{\epsilon}}(\pi_2^{1}) \otimes \sigma ,\omega_{\psi,W_1^{\epsilon'}} )\ne 0.$$ Thus we have constructed the unique pair $$\big(\Theta_{\psi, V_3^{\epsilon},W_1^{\epsilon'}}(\sigma) , \Theta_{\psi,V_2^{\epsilon}, W_{1}^{\epsilon'}}(\tau) \big) \in \Pi^{\epsilon}_{\theta^{(1)}(\phi_1)} \times \Pi^{\epsilon}_{\theta(\phi_2)} \text{ where } \epsilon=1, \epsilon'=\epsilon(\frac{1}{2},\phi_1 \otimes \chi^{-3},\psi_{2}^E )$$ which makes $\Hom_{U(V_2^{\epsilon})}(\pi_3^{1},\pi_2^{1})\ne0$ and using Theorem \ref{Tae} and Theorem \ref{thm1}, one can easily check that their characters of the component group are as in (\ref{phi}).\\

Next, we suppose that there is $(\sigma,\tau)\in \Pi^{\epsilon'}_{\phi_1} \times \Pi^{\epsilon''}_{\phi_2}$ such that $$\Hom_{U(V^{\epsilon}_{2})}\big(\Theta_{\psi, V_3^{\epsilon},W_1^{\epsilon'}}(\sigma) , \Theta_{V_2^{\epsilon}, W_{1}^{\epsilon''}}(\tau) \big)\ne 0.$$\\
By the see-saw identity, $\Hom_{U(W_1^{\epsilon'})}(\Theta_{\psi, W_1^{\epsilon'},V_2^{\epsilon}}(\Theta_{\psi,V_2^{\epsilon}, W_{1}^{\epsilon''}}(\tau)) \otimes \omega_{\psi,W_1^{\epsilon'}},\sigma )\ne0$ and so $$\Theta_{\psi, W_1^{\epsilon'},V_2^{\epsilon}}(\Theta_{\psi,V_2^{\epsilon}, W_{1}^{\epsilon''}}(\tau))\ne0.$$
If $\epsilon' \ne \epsilon'' $, then by Theorem \ref{Tae} (ii), we see that $\Theta_{V_2^{\epsilon},W_{1}^{\epsilon'}}(\Theta_{V_2^{\epsilon}, W_{1}^{\epsilon''}}(\tau))=0$ and so $\epsilon'=\epsilon''$. \\As in the above discussion, we also know that $\Pi_{\theta(\phi_2)}=\Pi^{1}_{\theta(\phi_2)}$ and so $\epsilon$ should be $1$. \\
Since $\Theta_{V_2^{\epsilon},W_{1}^{\epsilon'}}(\Theta_{V_2^{\epsilon}, W_{1}^{\epsilon'}}(\tau))=\tau \in \Pi^{\epsilon'}_{\phi_2}$ and $\Hom^{\vee}_{U(W_1^{\epsilon'})}(\tau \otimes \omega_{\psi,W_1^{\epsilon'}},\sigma )\simeq \Hom_{U(W_1^{\epsilon'})}(\tau^{\vee} \otimes \sigma, \omega_{\psi,W_1^{\epsilon'}} ),$  (FJ)$_1$ implies that $$\Hom_{U(W_1^{\epsilon'})}(\tau \otimes \omega_{\psi,W_1^{\epsilon'}},\sigma )\ne0 \Leftrightarrow \epsilon'=\epsilon(\frac{1}{2},\phi_1 \otimes \chi^{-3},\psi_{2}^E ).$$
However, by Theorem \ref{thm1}, $\epsilon \cdot \epsilon'=-\epsilon(\frac{1}{2},\phi_1 \otimes \chi^{-3},\psi_{2}^E )$ and so we have a contradiction. \\Thus we see that the recipe given in (B)$_2$ does not obtained by the theta lift from $U(W^{\epsilon'}_1)$.

\end{proof}

\section{\textbf{Proof of Theorem \ref{thm1}}}
In this section we prove Theorem \ref{thm1} using the $(FJ)_1$ and see-saw identity.\\

\noindent The first two statements on the existence of bijection and injection are immediate from Theorem \ref{theta} (ii).
\noindent To find the precise maps using the local Langlands correspondence, we consider the see-saw diagram: 
\[
 \xymatrix{
  \U(W_1^{\epsilon'})  \times \U(W_1^{\epsilon'})  \ar@{-}[dr] \ar@{-}[d] & \U(V^{\epsilon}_{3}) \ar@{-}[d] \\
  \U(W_1^{\epsilon'}) \ar@{-}[ur] &  \U(V^{\epsilon}_2) \times \U(L_{1})}.
\]
\\Recall that we are given an $L$-parameter $\phi$ for $U(W_1^{\pm})$ and $\pi=\pi(\phi,\eta)\in \Pi_{\phi}^{\epsilon'}$ where $\eta(a_1)=\epsilon'$. \\ Suppose $$\epsilon(\frac{1}{2},\phi \otimes \chi^{-3}, \psi_2^{E})=\epsilon \cdot \epsilon'.$$ By Theorem \ref{theta}, we know that $\Theta_{\psi,V_3^{\epsilon},W_1^{\epsilon'}}(\pi(\phi,\eta))$ is non-tempered and $\eta(a_1)=\epsilon'$, and $\theta_{1}(\eta)(a_1)=\epsilon$ and so the local theta correspondence would follow easily in this case.\\
 
On the other hand, we suppose $$\epsilon(\frac{1}{2},\phi \otimes \chi^{-3}, \psi_2^{E})=-\epsilon \cdot \epsilon'.$$ 

\noindent We choose an $L$-parameter $\phi_0 \ne \chi^2$ of $U(W_1^{\epsilon'})$ such that $$\epsilon(\frac{1}{2},\phi_{0}^{-1} \otimes  \phi \otimes \chi^{-1} ,\psi_2^E)=\epsilon'.$$
Since $\phi, \phi_0$ are unitary, for $\pi(\phi_0,\eta_0)\in \Pi_{\phi_0}^{\epsilon'}$, we know that $$\Hom_{U(W_{1}^{\epsilon'})}(\pi(\phi_0,\eta_0) \otimes \omega_{\psi,W_1^{\epsilon'}}, \pi(\phi,\eta))=\Hom_{U(W_1^{\epsilon'})}(\pi^{\vee}(\phi_0,\eta_0)\otimes \pi(\phi,\eta),\omega_{\psi,W_1^{\epsilon'}})\ne 0$$ (the second equality follows from the $(FJ)_1$ because the $L$-parameter of $\pi^{\vee}(\phi_0,\eta_0)$ is $\phi_0^{-1}$.) \\
Then by the see-saw idenity, we have
$$\Hom_{U(V_{2}^{\epsilon})}\big(\Theta_{\psi,V_3^{\epsilon},W_1^{\epsilon'}}(\pi(\phi,\eta)) ,\Theta_{\psi,V_2^{\epsilon},W_1^{\epsilon'}}(\pi(\phi_0,\eta_0))\big) \ne 0.$$ In particular, $$\Theta_{\psi,V_2^{\epsilon},W_1^{\epsilon'}}(\pi(\phi_0,\eta_0))\ne 0$$ and by Theorem \ref{Tae} (i), it has $L$-parameter $$\phi_0 \cdot \chi^{-1} \oplus \chi$$ and denote its component group by $$(\ZZ/2\ZZ)b_1 \times (\ZZ/2\ZZ)b_2.$$ 

Note that the theta lift $\Theta_{\psi,V_3^{\epsilon},W_1^{\epsilon'}}(\pi(\phi,\eta))$ is supercuspidal and since $U(W_1^{\epsilon'})$ is compact, it is a tempered representation. Furthermore since $\Theta_{\psi,V_2^{\epsilon},W_1^{\epsilon'}}(\pi(\phi_0,\eta_0))$ is tempered, (B)$_2$ implies

$$\theta_{2}(\eta)(a_1)=\epsilon(\frac{1}{2},\phi_{0}^{-1} \otimes  \phi \otimes \chi^{-1} ,\psi_{-2}^E) \cdot \epsilon(\frac{1}{2},\phi \otimes \chi^{-3}, \psi_{-2}^{E}),$$ 
$$\theta_{2}(\eta)(a_2)=\epsilon(\frac{1}{2},\phi_0^{-1}\otimes \chi^{2} \boxtimes \textbf{S}_2,\psi_{-2}^E)\cdot \epsilon(\frac{1}{2},\textbf{S}_2,\psi_{-2}^E),$$ $$\theta^{\vee}_{\psi,V_2^{\epsilon},W_1^{\epsilon'}}(\eta_0)(b_1)=\epsilon(\frac{1}{2},\phi_{0}^{-1} \otimes  \phi \otimes \chi^{-1} ,\psi_{-2}^E) \cdot \epsilon(\frac{1}{2},\phi_0^{-1}\otimes \chi^{2} \boxtimes \textbf{S}_2,\psi_{-2}^E),$$ $$\theta^{\vee}_{\psi,V_2^{\epsilon},W_1^{\epsilon'}}(\eta_0)(b_2)=\epsilon(\frac{1}{2},\phi \otimes \chi^{-3}, \psi_{-2}^{E}) \cdot \epsilon(\frac{1}{2},\textbf{S}_2,\psi_{-2}^E),$$ where $\psi_{-2}^E(x)=\psi^{E}(-2x).$ \\

Recall that for an $L$-parameter $\varphi$ of $U(V_n)$ or $U(W_n)$, $$\epsilon(\frac{1}{2},\varphi ,\psi_{-2}^E) = \big(\omega_{E/F}(-1)\big)^{n}\cdot \epsilon(\frac{1}{2},\varphi ,\psi_{2}^E).$$
Thus we have $$\epsilon(\frac{1}{2},\phi_{0}^{-1} \otimes  \phi \otimes \chi^{-1} ,\psi_{-2}^E) = \omega_{E/F}(-1)\cdot \epsilon(\frac{1}{2},\phi_{0}^{-1} \otimes  \phi \otimes \chi^{-1} ,\psi_{2}^E),$$ $$\epsilon(\frac{1}{2},\phi \otimes \chi^{-3}, \psi_{-2}^{E})=\omega_{E/F}(-1) \cdot \epsilon(\frac{1}{2},\phi \otimes \chi^{-3}, \psi_{2}^{E}),$$ and so we see $$\theta_{2}(\eta)(a_1)=\epsilon' \cdot \epsilon(\frac{1}{2},\phi \otimes \chi^{-3}, \psi_{2}^{E})=\eta(a_1) \cdot \epsilon(\frac{1}{2},\phi \otimes \chi^{-3}, \psi_{2}^{E}).$$\\
 
On the other hand, from Theorem \ref{Tae}, $$\theta^{\vee}_{V_2^{\epsilon},W_1^{\epsilon'}}(\eta_0)(b_1)=\omega_{E/F}(-1)\cdot \theta_{V_2^{\epsilon},W_1^{\epsilon'}}(\eta_0)(b_1)=\omega_{E/F}(-1) \cdot \epsilon(\frac{1}{2},\phi_{0}^{-1} \otimes  \phi \otimes \chi^{-1} ,\psi_{2}^E)=\epsilon(\frac{1}{2},\phi_{0}^{-1} \otimes  \phi \otimes \chi^{-1} ,\psi_{-2}^E)$$ and so we have $$\epsilon(\frac{1}{2},\phi_0^{-1}\otimes \chi^{2} \boxtimes \textbf{S}_2,\psi_{-2}^E)=1.$$ Since $$\theta_{\psi,V_2^{\epsilon},W_1^{\epsilon'}}(\eta_0)(b_1)\cdot \theta_{\psi,V_2^{\epsilon},W_1^{\epsilon'}}(\eta_0)(b_2)=\epsilon'$$ we see that
$$\theta^{\vee}_{\psi,V_2^{\epsilon},W_1^{\epsilon'}}(\eta_0)(b_2)=\omega_{E/F}(-1)\cdot \theta_{\psi,V_2^{\epsilon},W_1^{\epsilon'}}(\eta_0)(b_2)=\omega_{E/F}(-1)\cdot \epsilon \epsilon' = -\omega_{E/F}(-1)\cdot \epsilon(\frac{1}{2},\phi \otimes \chi^{-3}, \psi_{2}^{E})$$
 and it forces $$\epsilon(\frac{1}{2},\textbf{S}_2, \psi_{-2}^{E})=-1.$$ From these things, one can deduce $$\theta_{2}(\eta)(a_2)=-1$$ as we claimed.

\section{Ichino-Ikeda conjecture for the non-tempered case}
\noindent In this section, we give an analog of the Ichino-Ikeda conjecture for some non-tempered case using Theorem \ref{thm2}. To state our result, we review the notion of regularised local period introduced in \cite{Haan}.

\noindent  Let $E/F$ be a quadratic extension of number fields with $\mathbb{A}_E, \mathbb{A}_F$ their adele rings. \\
For a place $v$ of $F$, let $F_v$ be the completion of $F$ at $v$ and $E_v=E\otimes_F F_v$. For a (skew) hermitian space $V$ over $E$ and a place $v \in F$, consider $V_{v}:=V \otimes_F F_v$ as a (skew) hermitian space over $E_v$. Then we have a decomposition $$\mathbb{V}:=V \otimes_{F}\mathbb{A}_F \simeq \otimes_v V_{v}^{\epsilon_v} $$ where $v$ runs over all place of $F$ and $\epsilon_v:=\begin{cases}\epsilon(V_v) & \text{ for } v \text{ which remains prime in } E, \\ 1 & \text{ for } v \text{ which splits in } E \end{cases}. $

 With this decomposition, we have $$U(V)(\mathbb{A}_F)\simeq \Pi_v U(V^{\epsilon_v}_v)(F_v)$$ and if $\pi$ is an automorphic representation of $U(V)(\mathbb{A}_F)$, we also have its decomposition $\pi \simeq \otimes_v \pi_v$ where $\pi_v$ is an irreducible admissible representation of $U(V_v^{\epsilon_v})(F_v)$. 

Let $V_2$ be a 2-dimension hermitian space over $E$, $W_1$ be a 1-dimension skew-hermitian space over $E$ and $L$ be a hermitian line over $E$ with a form $\N_{E/F}$. Define $V_3:=V_2 \oplus L$. \\Note that $\epsilon(V_{3,v})=\epsilon(V_{2,v})$ for all $v$, because $\epsilon(L_v)=1$ for all $v$.

Let us see the following see-saw diagram 
\[
 \xymatrix{
  \U(W_1)  \times \U(W_1)  \ar@{-}[dr] \ar@{-}[d] & \U(V_{3}) \ar@{-}[d] \\
  \U(W_1) \ar@{-}[ur] &  \U(V_2) \times \U(L_{1})}.
\]
In this diagram, we consider the three local and global theta correspondence : 
\begin{enumerate}
\item $U(V_3) \times U(W_1)$ relative to the pair of characters $(\psi,\chi, \chi^3)$;
\item $\U(V_2) \times U(W_1)$ relative to the pair of characters $(\psi,\chi, \chi^2);$
\item $\U(L_{1}) \times \U(W_1)$ relative to the pair of characters $(\psi,\chi, \chi).$
\end{enumerate}
where $\psi,\chi$ are those we defined in Section \ref{not}.
For what follows, we suppress these choices from the notation.

Let $\sigma=\otimes_v \sigma_v $ be an automorphic character of $U(W_1)(\mathbb{A}_F)$ and $\mathbb{I}=\otimes_v \mathbb{I}_v$ be the trivial character of $U(W_1)(\mathbb{A}_F)$. Let $\pi_3:=\Theta_{V_{3},W_{1}}(\overline{\sigma})$ be the non-zero global theta lifts of $\overline{\sigma}$ to $U(V_3)(\mathbb{A}_F)$ and $\pi_2:=\Theta_{V_{2},W_{1}}(\overline{\mathbb{I}})$ be the non-zero global theta lifts of $\overline{\mathbb{I}}$ to  $U(V_2)(\mathbb{A}_F)$. (see Definition 2.2 in \cite{Haan})\\ For all place $v$ of $F$, if we set \begin{equation}\label{app}\pi_{3,v}:=\theta_{V_{3,v}^{\epsilon_v},W_{1,v}^{\epsilon'_v}}(\overline{\sigma_v}) ,\ \pi_{2,v}:=\theta_{V_{2,v}^{\epsilon_v},W_{1,v}^{\epsilon'_v}}(\overline{\mathbb{I}_v}) \text{ where } \epsilon_v=\epsilon(V_{3,v})=\epsilon(V_{2,v}) \text{ and } \epsilon'_v=\epsilon(W_{1,v}),\end{equation}
then by Howe duality, one has $\pi_{i} \simeq \otimes_v \pi_{i,v}$ for $i=2,3.$
\\Note that the two maps$$\theta_{V_{3,v},W_{1,v}} : \overline{\sigma_v} \otimes \omega_{V_{3,v},W_{1,v}}   \to \Theta_{V_{3,v}^{\epsilon_v},W_{1,v}^{\epsilon'_v}}(\overline{\sigma_v})\ ,
 \ \theta_{V_{2,v},W_{1,v}} : \overline{\mathbb{I}_v}  \otimes \omega_{V_{2,v},W_{1,v}} \to \Theta_{V_{2,v}^{\epsilon_v},W_{1,v}^{\epsilon'_v}}(\overline{\mathbb{I}_v})$$ are $U(V_{3})(F_v) \times U(W_{1})(F_v)$ and $U(V_{2})(F_v) \times U(W_{1})(F_v)$ equivariant surjective maps and by Theorem \ref{Tae} and \ref{theta}, the big theta lifts $\Theta_{V_{3,v}^{\epsilon_v},W_{1,v}^{\epsilon'_v}}(\overline{\sigma_v})$,
$\Theta_{V_{2,v}^{\epsilon_v},W_{1,v}^{\epsilon'_v}}(\overline{\mathbb{I}_v})$ are both irreducible. Thus we can define the local inner products $\mathcal{B}_{\pi_{i,v}}$ on $\pi_{i,v}$ for $i=2,3$ as follows:

For $\varphi_{1,v}^{i},\varphi_{2,v}^{i} \in \mathcal{S}(\mathbb{X}_{V_{i},W_{1}})(F_v)$, $f^{3}_{1,v},f^{3}_{2,v}\in \sigma_v$ and $f^{2}_{1,v},f^{2}_{2,v}\in \mathbb{I}_v$, let
$$\mathcal{B}_{\pi_{3,v}}(\theta_{v,V_{3},W_{1}}(\overline{f_{1,v}^{3}},\varphi_{1,v}^{3}),\theta_{v,V_{3},W_{1}}(\overline{f_{2,v}^{3}},\varphi_{2,v}^{3})):=\int_{U(W_1)(F_v)}\mathcal{B}_{\omega_{V_3,W_1}}(\omega_v(h_v)\cdot \varphi_{1,v}^3,\varphi_{2,v}^3)\cdot \mathcal{B}_{\sigma_v}(\sigma_v(h_v)\cdot f_{1,v}^3,f_{2,v}^3)dh_v$$ and 
$$\mathcal{B}_{\pi_{2,v}}(\theta_{v,V_{2},W_{1}}(\overline{f_{1,v}^{2}},\varphi_{1,v}^{2}),\theta_{v,V_{2},W_{1}}(\overline{f_{2,v}^{2}},\varphi_{2,v}^{2})):=\int_{U(W_1)(F_v)}\mathcal{B}_{\omega_{V_2,W_1}}(\omega_v(h_v)\cdot \varphi_{1,v}^2,\varphi_{2,v}^2)\cdot \mathcal{B}_{\mathbb{I}_v}(\mathbb{I}_v(h_v)\cdot f_{1,v}^2,f_{2,v}^2)dh_v$$ where $\mathcal{B}_{\omega_{V_i,W_1}}$ for $i=2,3$ are the local inner products of the Weil representations and $\mathcal{B}_{\sigma_v}$, $\mathcal{B}_{\mathbb{I}_v}$ are the local inner products of $\sigma_v, \mathbb{I}_v$ respectively.\\
With these choices of local inner products, Haan \cite{Haan} defined the regularised local period $\mathcal{P}_v^{reg}$ as follows; 

For $f_{3,v}\in \pi_{3,v}$ and $f_{2,v}\in \pi_{2,v}$, let
$$ {\mathcal{P}_v^{reg}(f_{3,v},f_{2,v})}:=c_v\cdot \lim_{s\to 0}\frac{\zeta_v(2s)}{L_v(s,BC(\pi_{2,v})\otimes\chi_v)}\cdot \int_{U(2)_v}\mathcal{B}_{\pi_{3,v}}(g_{v}\cdot f_{3,v},f_{3,v})\cdot \mathcal{B}_{\pi_{2,v}}(g_{v}\cdot f_{2,v},f_{2,v})\cdot \Delta(g_v)^s dg_{v}.$$
(here, $c_v$ is a non-zero constant for each $v$ and $\Delta(g_v)$ is some determinant map appearing in the doubling method. For the precise definition of $c_v$ and $\Delta(g_v)$, we refer the reader to Section 3.2 in \cite{Haan}.)

By \cite[Remark 4.1]{Haan}, $\mathcal{P}_v^{reg}\ne 0$ is equivalent to $\Theta_{W_{1,v},L_{1,v}}(\overline{\sigma_v})\ne 0.$
From this observation, we have the following theorem.

\begin{thm} Using the notation as in the (\ref{app}), for non-archimedean place $v$, we have $$\Hom_{U(V_2)(F_v)}(\pi_{3,v},\pi_{2,v})\ne0  \Leftrightarrow \mathcal{P}_v^{reg} \ne 0$$ 
\end{thm}

\begin{proof}If $v$ is split, all relevent groups are general linear groups and so we consider the following see-saw diagram:
\[
 \xymatrix{
  \GL(W_1)  \times \GL(W_1)  \ar@{-}[dr] \ar@{-}[d] & \GL(V_{3}) \ar@{-}[d] \\
  \GL(W_1) \ar@{-}[ur] &  \GL(V_2) \times \GL(L_{1})}
\]
Thus by the see-saw identity, $$\Hom_{GL(V_2)(F_v)}(\pi_{3,v},\pi_{2,v})\simeq \Hom_{GL(W_1)}(\overline{\mathbb{I}_v} \otimes \omega_{v,W_1,L_1},\overline{\sigma_v}).$$ From Theorem 17.2 in \cite{Gan2}, $$\Hom_{GL(W_1)}(\overline{\mathbb{I}_v} \otimes \omega_{v,W_1,L_1},\overline{\sigma_v})\ne 0$$ and so one has $$\Hom_{GL(V_2)(F_v)}(\pi_{3,v},\pi_{2,v})\ne0.$$ 
On the other hand, if one follows the similar argument as in Proposition 2.6.1 in \cite{Ge2}, one can have $\Theta_{v,W_1,L_1}(\overline{\sigma_v})\ne0$. (Indeed, it is known that the theta lift from $GL_n(F_v)$ to $GL_n(F_v)$ is just taking a representation to its contragradient representation.) Thus we see that the theorem holds for split places $v$.

Next, suppose that $v$ remains prime in $E$. Then by Remark \ref{rem1}, we have $$\Hom_{U(V_2)(F_v)}(\pi_{3,v},\pi_{2,v})\ne0 \Leftrightarrow \epsilon(W_{1,v})=\epsilon(\frac{1}{2},\sigma^{-1}\otimes \chi^{-1} ,\psi_2^E).$$
On the other hand, from Theorem \ref{Te}, one has $$\Theta_{W_{1,v},L_{1,v}}(\overline{\sigma_v})\ne0 \Leftrightarrow \epsilon(W_{1,v})=\epsilon(\frac{1}{2},\sigma^{-1}\otimes \chi^{-1} ,\psi_2^E).$$ Thus we verified our claim for primes $v$ which are inert in $E$.

\end{proof}

\end{document}